\documentclass[12pt]{amsart}
\usepackage{amscd, amsfonts, amssymb}
\usepackage{fullpage}
\usepackage{latexsym}
\usepackage{verbatim}
\usepackage{enumerate}

\newcommand{\tuple}[1]{{\mathbf #1}}

\newcommand\ra{\rightarrow}

\newcommand{\NN}{{\mathbb N}}

\newcommand{\FF}{{\mathbb F}}

\numberwithin{equation}{section}

\newtheorem{thm}[equation]{Theorem}

\newtheorem{lem}[equation]{Lemma}
\newtheorem{cor}[equation]{Corollary}
\newtheorem{prop}[equation]{Proposition}

\theoremstyle{definition}
\newtheorem{defn}[equation]{Definition}
\newtheorem{exmp}[equation]{Example}

\theoremstyle{remark}
\newtheorem{rem}[equation]{Remark}
\theoremstyle{remark}

\newcommand{\ovl}{\overline}

\subjclass[2010]{14L30 (20G15)}
\keywords{Quasi-projective $G$-varieties; generic stabilisers; principal orbit type; $G$-complete reducibility}

\date{October 8, 2015}

\title[Generic stabilisers for actions of reductive groups]
{Generic stabilisers for actions of reductive groups}

\dedicatory{In memory of Robert Steinberg}

\author[B.\ Martin]{Benjamin Martin}
\address
{Department of Mathematics,
University of Aberdeen,
King's College,
Fraser Noble Building,
Aberdeen AB24 3UE,
United Kingdom}
\email{b.martin@abdn.ac.uk}

\begin{document}

\begin{abstract}
Let $G$ be a reductive algebraic group over an algebraically closed field and let $V$ be a quasi-projective $G$-variety.  We prove that the set of points $v\in V$ such that ${\rm dim}(G_v)$ is minimal and $G_v$ is reductive is open.  We also prove some results on the existence of principal stabilisers in an appropriate sense.
\end{abstract}

\maketitle

\section{Introduction}
\label{sec:intro}

Let $G$ be a reductive linear algebraic group over an algebraically closed field $k$ and let $V$ be a quasi-projective $G$-variety.  For convenience, we assume throughout the paper that $G$ permutes the irreducible components of $V$ transitively (the extension of our results to the general case is straightforward).  An important question in geometric invariant theory is the following: what can we say about generic stabilisers for the $G$-action?  For instance, given $v\in V$, what does the stabiliser $G_v$ tell us about the stabilisers $G_w$ for $w$ near $v$?  Define $V_0$ to be the set of points $v\in V$ such that the stabiliser $G_v$ has minimal dimension.  The basic theory tells us that $V_0$ is open (Lemma~\ref{lem:stabdimcty}).  Here is a deeper result \cite[Prop.\ 8.6]{BR}: if $V$ is affine and there exists an \'etale slice through $v$ for the $G$-action then there exists an open neighbourhood $U$ of $v$ such that for all $w\in U$, $G_w$ is conjugate to a subgroup of $G_v$.  In particular, if ${\rm dim}(G_v)$ is minimal in this case then $G_w^0$ is conjugate to $G_v^0$ for all $w\in U$.  The existence of an \'etale slice requires, among other conditions, that $V$ be affine and the orbit $G\cdot v$ be closed and separable.  If $V$ is affine and $k$ has characteristic 0 then every $v\in V$ such that $G\cdot v$ is closed admits an \'etale slice, but if $k$ has positive characteristic then it can happen that there are no \'etale slices at all, since, for example, orbits need not be separable. 

In this paper we prove some results about properties of generic stabilisers.  Most previous work in this area has dealt with affine varieties and/or fields of characteristic zero only.  Our results hold for quasi-projective varieties and in arbitrary characteristic, although in some cases we get stronger results in characteristic zero.   We need no assumptions on the existence of or properties of closed orbits, and we allow $G$ to be non-connected.
 
Let $V_{\rm red}= \{v\in V_0\mid\mbox{$G_v$ is reductive}\}$.  It is possible for $V_{\rm red}$ to be empty (see Example~\ref{exmp:unipotent}).  Our first main result implies that if $V_{\rm red}$ is nonempty then generic stabilisers are reductive.

\begin{thm}
\label{thm:main}
 $V_{\rm red}$ is an open subvariety of $V$.
\end{thm}

\noindent A key ingredient in the proof is the Projective Extension Theorem (see Lemma~\ref{lem:projextn}).

We mention two related results.  First, it follows from \cite[Cor.~9.1.2]{rich72acta} that if $G$ is a complex linear algebraic group---not necessarily reductive---and $V$ is a smooth algebraic transformation space for $G$ then $V_{\rm red}$ is open.  Second, V. Popov proved the following \cite{pop} (cf.\ \cite{LunaVust}).  Let $G$ be a connected linear algebraic group---not necessarily reductive, and in arbitrary characteristic---such that $G$ has no nontrivial characters, and let $V$ be an irreducible normal algebraic variety on which $G$ acts such that the divisor class group ${\rm Cl}(V)$ has no elements of infinite order.  Then generic orbits $G$-orbits on $V$ are closed if generic $G$-orbits on $V$ are affine, and the converse also holds if $V$ is affine.

Richardson proved that if $G$ is reductive and $V$ is an affine $G$-variety then an orbit $G\cdot v$ is affine if and only if the stabiliser $G_v$ is reductive \cite[Thm.~A]{Rich77}.  Suppose $V$ is affine and there exists a closed orbit $G\cdot v$ of maximum dimension; then the union of the closed orbits of maximal dimension is open in $V$ \cite[Prop.\ 3.8]{New}.  It follows from Richardson's result that there is an open dense set of points $v\in V$ such that $G_v$ is reductive.  Theorem~\ref{thm:main} extends this to the case when generic orbits are not closed, without the affineness assumption.

Richardson's result discussed above gives the following immediate corollary to Theorem~\ref{thm:main} (note that $G_v$ has minimal dimension if and only if the orbit $G\cdot v$ has maximum dimension).

\begin{cor}
\label{cor:affine}
 Suppose $V$ is affine.  Then the set $v\in V$ such that ${\rm dim}(G\cdot v)$ is maximal and $G\cdot v$ is affine is open.
\end{cor}

We give an application of Theorem~\ref{thm:main}.  Nisnevi{\v{c}} \cite{nisnevic} proved the following result when ${\rm char}(k)= 0$ and $t= 1$\footnote{Wallach also has a proof in this case \cite{wallach}.}; he also proved that the subset $A$ is nonempty in this special case.

\begin{thm}
\label{thm:subgpint}
 Let $M,H_1,\ldots, H_t$ be subgroups of a reductive group $G$ such that $M$ is reductive.  Let
 $$ A= \{(g_1,\ldots, g_t)\in G^t\mid \mbox{$M\cap g_1H_1g_1^{-1}\cap \cdots \cap g_tH_tg_t^{-1}$ is reductive and has minimal dimension}\}. $$
 Then $A$ is open.
\end{thm}

\noindent We do not know in general whether $A$ can be empty in positive characteristic, not even when $t= 1$ and $H_1= M$.

If generic stabilisers are reductive, it is reasonable to try to pin down which reductive subgroups of $G$ actually appear as stabilisers.  We say that a subgroup $H$ of $G$ is a {\em principal stabiliser} for the $G$-variety $V$ if there is a nonempty open subset $O$ of $V$ such that $G_v$ is conjugate to $H$ for all $v\in O$.  We then say that $V$ has a {\em principal orbit type}.  Under our assumptions on $G$ and $V$, a principal stabiliser is unique up to conjugacy if it exists.  Richardson proved that if ${\rm char}(k)= 0$ and $V$ is smooth and affine then a principal stabiliser exists \cite[Prop.\ 5.3]{rich72}.

It turns out that in positive characteristic, the condition of conjugacy of the stabilisers is too strong: Example~\ref{exmp:notprinc} below shows that even if generic stabilisers are connected and reductive, a principal stabiliser need not exist.  To obtain a result, we need to weaken the notion of principal stabiliser.  Let $M\leq G$ and let $P$ be a minimal R-parabolic subgroup of $G$ containing $M$ (see Section~\ref{sec:prelims} for the definition of R-parabolic subgroups), let $L$ be an R-Levi subgroup of $P$ and let $\pi_L\colon P\ra L$ be the canonical projection.  It can be shown that up to $G$-conjugacy, $\pi_L(M)$ does not depend on the choice of $P$ and $L$ (cf.\ \cite[Prop.\ 5.14(i)]{GIT}).  We define ${\mathcal D}(M)$ to be the conjugacy class $G\cdot \pi_L(M)$, and we call this the {\em $G$-completely reducible degeneration} of $M$ (see Section~\ref{sec:Gcr} for the definition of $G$-complete reducibility).  Our second main result says that the ${\mathcal D}(G_v)$ are equal for generic $v$.

\begin{thm}
\label{thm:genericstab}
 There exist a $G$-completely reducible subgroup $H$ of $G$ and a nonempty open subset $O$ of $V$ such that ${\mathcal D}(G_v)= G\cdot H$ for all $v\in O$.
\end{thm}

If $G$ is connected and every stabiliser is unipotent then ${\mathcal D}(G_v)= 1$ for all $v\in V$, so we don't learn much about the structure of the stabilisers.  Under some extra hypotheses, however, we can deduce the existence of a principal stabiliser.

\begin{cor}
\label{cor:Gcrprinc}
 Suppose there is a nonempty open subset $O$ of $V$ such that $G_v$ is $G$-completely reducible for all $v\in O$.  Then the subgroup $H$ from Theorem~\ref{thm:genericstab} is a principal stabiliser for $V$.
\end{cor}

\begin{cor}
\label{cor:char0princ}
 Suppose ${\rm char}(k)= 0$ and $V_{\rm red}$ is nonempty.  Then the subgroup $H$ from Theorem~\ref{thm:genericstab} is a principal stabiliser for $V$.
\end{cor}


If we restrict ourselves to the identity components of stabilisers then we get slightly stronger results.

\begin{thm}
\label{thm:genericstab_weak_conn}
 Suppose $V_{\rm red}$ is nonempty.  There exists a connected $G$-completely reducible subgroup $H$ of $G$ such that ${\mathcal D}(G_v^0)= G\cdot H$ for all $v\in V_{\rm red}$.
\end{thm}

\noindent In fact, we prove a version of Theorem~\ref{thm:genericstab_weak_conn} which applies even when $V_{\rm red}$ is empty (see Theorem~\ref{thm:genericstab_conn}).

We briefly explain our approach to the proof of Theorems~\ref{thm:genericstab} and \ref{thm:genericstab_weak_conn}.  We may regard the subgroups $G_v$ as a family of subgroups of $G$ parametrised by $V$.  There is no obvious way to endow a set of subgroups of $G$ with a geometric structure, so instead we follow the approach of R.W.\ Richardson \cite{rich2}, \cite{rich} and consider the set of tuples that generate these subgroups.

\begin{defn}
 Let $N\in {\mathbb N}$.  Define $C= C_N= \{(v,g_1,\ldots, g_N)\mid v\in V, g_1,\ldots, g_N\in G_v\}$.  We call $C$ the {\em stabiliser variety} of $V$.
\end{defn}

\noindent Our results follow from a study of the geometry of $C$, using the theory of character varieties and the theory of $G$-complete reducibility.  A major technical problem is that $C$ can be reducible even when $G$ is connected and $V$ is irreducible, so the projection into $V$ of a nonempty open subset of $C$ need not be dense (see Remarks~\ref{rem:genericGirstab} and \ref{rem:notgeneric}, for example).  The situation is better if we consider only the identity components of stabilisers: we can work with a canonically defined subvariety $\widetilde{C}$ of $C$ with nicer properties (see Lemma~\ref{lem:cpts}).

The paper is laid out as follows.  Section~\ref{sec:prelims} contains preliminary material.  In Section~\ref{sec:mainthm} we prove Theorems~\ref{thm:main} and \ref{thm:subgpint}.  Section~\ref{sec:Gcr} reviews $G$-complete reducibility and Section~\ref{sec:preceq} introduces a technical tool needed in Section~\ref{sec:genericstab}, where we prove Theorem~\ref{thm:genericstab} and Corollaries~ \ref{cor:Gcrprinc} and \ref{cor:char0princ}.  We study the irreducible components of $C$ in Section~\ref{sec:stabcpts} and prove Theorem~\ref{thm:genericstab_weak_conn}.  The final section contains some examples.

\section{Preliminaries}
\label{sec:prelims}

Throughout the paper, $N$ denotes a positive integer, $G$ is a reductive linear algebraic group---possibly non-connected---over an algebraically closed field $k$ and $V$ is a quasi-projective $G$-variety over $k$.  The stabiliser variety $C_N$ depends on the choice of $N$, but to ease notation we suppress the subscript and write just $C$.  All subgroups of $G$ are assumed to be closed.  If $H$ is a linear algebraic group then we write $\kappa(H)$ for the number of connected components of $H$, $R_u(H)$ for the unipotent radical of $H$ and $\alpha_H$ for the canonical projection $H\ra H/R_u(H)$.  The irreducible components of $H^N$ are the subsets of the form $H_1\times\cdots \times H_N$, where each $H_i$ is a connected component of $H$.  If $X'$ is a subset of a variety $X$ then we denote the closure of $X'$ in $X$ by $\overline{X'}$.  Below we will use the following results on fibres of morphisms (cf.\ \cite[AG.10.1 Thm.]{Bo}): if $f\colon X\ra Y$ is a dominant morphism of irreducible quasi-projective varieties then for all $y\in Y$, every irreducible component of $f^{-1}(y)$ has dimension at least ${\rm dim}(X)- {\rm dim}(Y)$, and there is a nonempty open subset $U$ of $Y$ such that if $y\in U$ then equality holds.  More generally, if $Z$ is a closed irreducible subset of $Y$ and $W$ is an irreducible component of $f^{-1}(Z)$ that dominates $Z$ then ${\rm dim}(W)\geq {\rm dim}(Z)+ {\rm dim}(X)- {\rm dim}(Y)$.

The next result is \cite[Lem.~3.7]{New}.

\begin{lem}
\label{lem:stabdimcty}
 Let a linear algebraic group $H$ act on a quasi-projective variety $W$.  For any $t\in \NN\cup \{0\}$, the set $\{w\in W\mid {\rm dim}(H_w)\geq t\}$ is closed.
\end{lem}

Our assumption that $G$ permutes the irreducible components of $V$ transitively implies that these components all have the same dimension, which we denote by $n$, and also that nonempty open $G$-stable subsets of $V$ are dense.  In particular, the open subset $V_0$ is dense; we denote the dimension of $G_v$ for any $v\in V_0$ by $r$.

The group $G$ acts on $G^N$ by simultaneous conjugation: $g\cdot (g_1,\ldots, g_N)= (gg_1g^{-1},\ldots, gg_Ng^{-1})$.  We define $\phi\colon C\ra G^N$ and $\eta\colon C\ra V$ to be the canonical projections.  We allow $G$ to act on $C$ in the obvious way, so that $\phi$ and $\eta$ are $G$-equivariant. 


We recall an approach to parabolic subgroups and Levi subgroups using cocharacters \cite[Sec.~8.4]{spr2}, \cite[Lem.~2.4, Sec.\ 6]{BMR}.  We denote by $Y(G)$ the set of cocharacters of $G$.  The subgroup $P_\lambda:= \{g\in G\mid \lim_{a\ra 0} \lambda(a)g\lambda(a)^{-1}\ \mbox{exists}\}$ is called an {\em R-parabolic subgroup} of $G$, and the subset $L_\lambda:= C_G(\lambda(k^*))$ is called an {\em R-Levi subgroup} of $P_\lambda$.  An R-parabolic subgroup $P$ is parabolic in the sense that $G/P$ is complete, and $P^0$ is a parabolic subgroup of $G^0$.  If $G$ is connected then an R-parabolic (resp.\ R-Levi) subgroup is a parabolic (resp.\ Levi) subgroup, and every parabolic subgroup $P$ and every Levi subgroup $L$ of $P$ arise in this way.  The normaliser $N_G(P)$ of a parabolic subgroup $P$ of $G^0$ is an R-parabolic subgroup.  The subset $\{g\in G\mid \lim_{a\ra 0} \lambda(a)g\lambda(a)^{-1}= 1\}$ is the unipotent radical $R_u(P_\lambda)$, and this coincides with $R_u(P_\lambda^0)$.  We denote the canonical projection from $P_\lambda$ to $L_\lambda$ by $c_\lambda$.  There are only finitely many conjugacy classes of R-parabolic subgroups \cite[Prop.\ 5.2(e)]{Mar}.

We finish with some results that are well known; we give proofs here as we could not find any in the literature.  These results are not needed in the proofs of Theorems~\ref{thm:main} and \ref{thm:subgpint}. 

\begin{lem}
\label{lem:bddfib}
 Let $\psi\colon X\ra Y$ be a morphism of quasi-projective varieties over $k$.  There exists $d\in {\mathbb N}$ such that any fibre of $\psi$ has at most $d$ irreducible components.
\end{lem}

\begin{proof}
 By noetherian induction on closed subsets of $X$ and $Y$, we are free to pass to open affine subvarieties of $X$ and $Y$ whenever this is convenient.  So assume that $X$, $Y$ are affine and let $R$, $S$ be the co-ordinate rings of $X$, $Y$ respectively.  Suppose first that
 $X$ and $Y$ are irreducible and that $\psi$ is finite and dominant.  By a simple induction argument, we can assume that $R=S[f]$ for some $f\in R$.  Let $m(t)=t^d+ a_{d-1}t^{d-1}+\cdots +a_0$ be the minimal polynomial of $f$ with respect to the quotient field of $S$.  Passing to open subvarieties, we can assume that the $a_i$ are defined on $Y$.  Let $y\in Y$.  If $x\in X$ with $\psi(x)=y$ then we have $f(x)^d+ a_{d-1}(y)f(x)^{d-1}+\cdots +a_0(y)=0$; it follows that there can be at most $d$ such values of $x$.  Thus the fibres of $\psi$ have cardinality at most $d$.

Now consider the general case.  Passing to open subvarieties, we can assume that $X$, $Y$ are irreducible and affine and that $\psi$ is dominant.  We can write $R=S[f_1,\ldots,f_t]$ for some $t$ and some $f_1,\ldots,f_t\in R$.  After reordering the $f_i$ if necessary, there exists $s$ with $0\leq s\leq t$ such that $f_1,\ldots,f_s$ are algebraically independent over $S$ and $f_{s+1},\ldots,f_t$ are algebraic over $S[f_1,\ldots,f_s]$.  The inclusion $S\subseteq S[f_1,\ldots,f_s]\subseteq R$ corresponds to a factorisation of $\psi$ as $\psi=X\stackrel{\psi'}{\ra} Y'\stackrel{g}{\ra} Y$, where $Y'$ is the affine variety with co-ordinate ring $S[f_1,\ldots,f_s]$.  Then ${\rm dim}(X)= {\rm dim}(Y')$ and $\psi'$ is dominant.  By passing to open affine subvarieties, we can assume that $\psi'$ is finite and $Y'$ is normal.  By the special case above, the cardinality of the fibres of $\psi'$ is bounded by some $d$.

Suppose that for some $y\in Y$, the fibre $F:=\psi^{-1}(y)$ has $d+1$ distinct irreducible components, say $F_1,\ldots,F_{d+1}$.  The fibre $F':=g^{-1}(y)$ is clearly isomorphic to $k^s$ and we have $F=(\psi')^{-1}(F')$.  Since $\psi'$ is finite and $Y'$ is normal, every irreducible component of $F$ has dimension $s$ and is mapped surjectively to $F'$ \cite[4.2 Prop.\ (b)]{Hum}.  But this means that for generic $y'\in F'$, $(\psi')^{-1}(y')$ has at least $d+1$ elements, a contradiction.  We deduce that $F$ has at most $d$ irreducible components, as required.
\end{proof}

\begin{defn}
 Applying Lemma~\ref{lem:bddfib} to the map $\eta\colon C\ra V$, we see there is a uniform bound on $\kappa(G_v)$ as $v$ ranges over the elements of $V$, since the number of irreducible components of $G_v^N$ is $\kappa(G_v)^N$.  We denote the least such bound by $\Theta$.
\end{defn}

\begin{lem}
\label{lem:refined_baire}
 Let $\Omega/k$ be a proper extension of algebraically closed fields.  Let $t\in {\mathbb N}$ and let $X$ be an $\Omega$-defined constructible subset of $\Omega^t$.  Let $\{X_i\mid i\in I\}$ be a family of $k$-defined constructible subsets of $\Omega^t$ such that $X\subseteq \bigcup_{i\in I} X_i$.  Then there exists $i\in I$ such that $X\cap X_i$ has nonempty interior in $X$.  Moreover, there exists a finite subset $F$ of $I$ such that $X\subseteq \bigcup_{i\in F} X_i$.
\end{lem}

\begin{proof}
 Clearly we can reduce to the case when $X$ and each of the $X_i$ is irreducible and locally closed in $\Omega^t$.  The second assertion follows from the first by Noetherian induction on closed subsets of $X$, so it is enough to prove the first assertion.  Let $m= {\rm dim}(X)$.  It suffices to show that ${\rm dim}(X\cap X_i)= m$ for some $i\in I$.  We use induction on $m$.  The result is trivial if $m= 0$.  Choose polynomials $f_1,\ldots, f_m\in k[T_1,\ldots, T_t]$ such that the restrictions of the $f_i$ to $X$ form a subset of the co-ordinate ring $\Omega[X]$ that is algebraically independent over $\Omega$.  Define $f\colon \Omega^t\ra \Omega^m$ by $f(x)= (f_1(x),\ldots, f_m(x))$; note that $f$ is $k$-defined.  Any proper closed subset of $X$ is a union of irreducible components of dimension less than $m$.  By induction, we are therefore free to replace $X$ with any nonempty open subset of $X$, so we can assume that $f|_{X}$ gives a finite map from $X$ onto an open subset of $\Omega^m$.  Then $f(X)\subseteq \bigcup_{i\in I} f(X_i)$ and each $f(X_i)$ is $k$-constructible.  It is enough to prove that $f(X)\cap f(X_i)$ has nonempty interior in $f(X)$.   Hence we can assume without loss that $t= m$ and $X$ is an open subset of $\Omega^m$.
 
 Let $\pi\colon \Omega^m\ra \Omega$ be the projection onto the first co-ordinate.  Since $X$ is an open and dense subset of $\Omega^m$, $\pi(X)$ is a dense constructible subset of $\Omega$, so $\Omega\backslash \pi(X)$ is finite.  Hence there exists $y\in \pi(X)$ such that $y\not\in k$.  Let $\widetilde{X}= X\cap \pi^{-1}(y)$.  Then $\widetilde{X}$ is an $\Omega$-defined locally closed subset of $\Omega^m$,  $\widetilde{X}$ is irreducible of dimension $m- 1$ and $\widetilde{X}\subseteq \bigcup_{i\in I} X_i$.  By induction, there exists $j\in I$ such that $\widetilde{X}\cap X_j$ has an irreducible component of dimension $m-1$.  Hence $\pi^{-1}(y)\cap X_j$ has an irreducible component of dimension at least $m-1$.  Note that we retain our assumption that the $X_i$ are irreducible.  Now $X_j$ cannot be contained in $\pi^{-1}(y)$ because $\pi^{-1}(y)$ has no $k$-points, so $\pi^{-1}(y)\cap X_j$ is a proper closed subset of $X_j$.  Hence ${\rm dim}(X_j)= m$, as required.
\end{proof}

\begin{cor}
\label{cor:qcmpct}
 Let $\Omega$ be an uncountable algebraically closed field.  Let $t\in {\mathbb N}$ and let $X$ be an $\Omega$-defined constructible subset of $\Omega^t$.  Let $\{X_i\mid i\in I\}$ be a countable family of $\Omega$-constructible subsets of $X$ such that $X\subseteq \bigcup_{i\in I} X_i$.  Then there exists $i\in I$ such that $X_i$ has nonempty interior in $X$.  Moreover, there exists a finite subset $F$ of $I$ such that $X\subseteq \bigcup_{i\in F} X_i$.
\end{cor}

\begin{proof}
 Each of the $X_i$ is defined over a subfield of $\Omega$ that is finitely generated over the algebraic closure of the prime field, so there exists a countable subfield $k$ of $\Omega$ such that each of the $X_i$ is defined over $k$.  Since $k$ is countable and $\Omega$ is not, $\Omega/k$ is a proper field extension.  Now apply Lemma~\ref{lem:refined_baire}.
\end{proof}

\begin{cor}
\label{cor:qcmpctirred}
 If $X$ is irreducible and the $X_i$ are closed in Corollary~\ref{cor:qcmpct} then there exists $i\in I$ such that $X\subseteq X_i$.
\end{cor}

\begin{proof}
 This is immediate from Corollary~\ref{cor:qcmpct}.
\end{proof}

\section{Proof of Theorem~\ref{thm:main}}
\label{sec:mainthm}

We now prove our first main result.

\begin{lem}
\label{lem:projextn}
 Let $X$ be a quasi-projective variety, let $Y$ be a projective variety and let $Z$ be a closed subvariety of $X\times Y$.  Then the projection of $Z$ onto $X$ is a closed variety.
\end{lem}

\begin{proof}
 Choose a covering of $X$ by open affine subvarieties $X_1,\ldots, X_m$.  A subset $S$ of $X$ (resp.\ $X\times Y$) is closed if and only if its intersection with $X_i$ (resp.\ $X_i\cap Y$) is closed for all $i$, so we can assume that $X$ is affine.  The result now follows from the Projective Extension Theorem \cite[Ch.\ 8, Sec.\ 5, Thm.\ 6]{clo}.
\end{proof}

\begin{lem}
\label{lem:closedcrit}
 Let $P$ be an R-parabolic subgroup of $G$ and let $W$ be a closed $P$-stable subset of $V$.  Then $G\cdot W$ is closed in $V$.
\end{lem}

\begin{proof}
 Set $D= \{(v,g)\in V\times G\mid g^{-1}\cdot v\in W\}$, a closed subvariety of $V\times G$.  We let $P$ act on $V\times G$ by $h\cdot (v,g)= (v,gh^{-1})$; then $D$ is $P$-stable as $W$ is.  Let $\pi_P\colon G\ra G/P$ be the canonical projection and define $\theta\colon V\times G\ra V\times G/P$ by $\theta(v,g)= (v,\pi_P(g))$.  Since $\pi_P$ is smooth, $\pi_P$ is flat, so $(\theta,V\times G/P)$ is a geometric quotient by \cite[Lem.\ 5.9(a)]{bongartz}.  Then $\theta$ takes closed $P$-stable subvarieties of $V\times G$ to closed subvarieties of $V\times G/P$, so $\theta(D)$ is a closed subvariety of $V\times G/P$.  Note that the projection of $\theta(D)$ onto $V$ is $G\cdot W$.  Lemma~\ref{lem:projextn} implies that $G\cdot W$ is closed in $V$, so we are done. 
\end{proof}

\begin{rem}
\label{rem:Pclosed}
 We record one corollary (cf.\ \cite[Prop.\ 27]{sikora}).  Recall that $G$ acts on $G^N$ by simultaneous conjugation.  Let $P$ be an R-parabolic subgroup of $G$.  Then $G\cdot P^N$ is closed in $G^N$.  This follows immediately from Lemma~\ref{lem:closedcrit}, taking $V= G^N$ and $W= P^N$.
\end{rem}

\begin{prop}
\label{prop:redcrit}
 Let $P$ be an R-parabolic subgroup of $G$ with unipotent radical $U$.  Set $V_P= \{v\in V_0\mid {\rm dim}(P_v)= r\}= \{v\in V_0\mid G_v^0\leq P\}$ and $V_{P,t}= \{v\in V_P\mid {\rm dim}(U_v)\geq t\}$.  Then $G\cdot V_{P,t}$ is closed in $V_0$ for each $t$.
\end{prop}

\begin{proof}
 This follows from Lemma~\ref{lem:closedcrit} (applied to $V_0$), as each $V_{P,t}$ is $P$-stable and closed in $V_0$ (Lemma~\ref{lem:stabdimcty}).
\end{proof}

\begin{proof}[Proof of Theorem~\ref{thm:main}]
 We show that $G_v$ is non-reductive if and only if $v\in \bigcup_P G\cdot V_{P,1}$, where the union is over a set of representatives of the conjugacy classes of R-parabolic subgroups of $G$.  Since there are only finitely many R-parabolic subgroups up to conjugacy and each subset $G\cdot V_{P,1}$ is closed in $V_0$ (Proposition~\ref{prop:redcrit}), this suffices to prove the theorem.
 
 If $v\in G\cdot V_{P,1}$---say, $g\cdot v\in V_{P,1}$---then $G_v^0\leq g^{-1}Pg$ and $G_v^0$ contains a positive-dimensional subgroup $M$ of $g^{-1}Ug= R_u(g^{-1}Pg)$, so $G_v^0$ is not reductive: for $G_v^0$ normalises the connected unipotent subgroup of $g^{-1}Ug$ generated by the $G_v^0$-conjugates of $M$.  Hence $G_v$ is not reductive, either.  Conversely, if $v\in V_0$ and $G_v$ has nontrivial unipotent radical $H$ then we can pick a minimal R-parabolic subgroup $P$ of $G$ containing $G_v$; then $H\leq R_u(P)$ (see the paragraph following Lemma~\ref{lem:countable}), so $v\in G\cdot V_{P,1}$.  The result now follows.
\end{proof}

\begin{rem}
\label{rem:Vmin}
 More generally, set $V(t)= \{v\in V_0\mid {\rm dim}(R_u(G_v))\geq t\}$.  A similar argument to the one above shows that $V(t)= \bigcup_P G\cdot V_{P,t}$, where the union is over a set of representatives of the conjugacy classes of R-parabolic subgroups of $G$, so $V(t)$ is closed.  In particular, define $V_{\rm min}= \{v\in V_0\mid {\rm dim}(R_u(G_v))\ \mbox{is minimal}\}$; then $V_{\rm min}$ is a nonempty open subset of $V_0$.  Note that $V_{\rm min}= V_{\rm red}$ if $V_{\rm red}$ is nonempty.
\end{rem}

We finish the section with the proof of Theorem~\ref{thm:subgpint}.  Each coset space $G/H_i $ is quasi-projective, and the reductive group $M$ acts on $G/H_i$ by left multiplication.  Let $V= G/H_1\times\cdots \times G/H_t$, with $M$ acting on $V$ by the product action.  For any $(g_1,\ldots, g_t)\in G^t$, the stabiliser $M_{(g_1H_1,\ldots, g_tH_t)}$ is equal to $M\cap g_1H_1g_1^{-1}\cap \cdots \cap g_tH_tg_t^{-1}$.  Hence the set $A$ equals the preimage of $V_{\rm red}$ under the map from $G^t$ to $V$ that sends $(g_1,\ldots, g_t)$ to $(g_1H_1,\ldots, g_tH_t)$.  But $V_{\rm red}$ is open by Theorem~\ref{thm:main}, so $A$ is open.  This completes the proof.

\begin{rem}
 In the set-up in the proof of Theorem~\ref{thm:subgpint}, we do not know whether the subgroups $M\cap g_1H_1g_1^{-1}\cap \cdots \cap g_tH_tg_t^{-1}$ are all conjugate for generic $(g_1,\ldots, g_t)$.  This is the case, however, if these subgroups are $G$-completely reducible for generic $(g_1,\ldots, g_t)$ (cf.\ Example~\ref{exmp:guralnick}).
\end{rem}

\section{$G$-complete reducibility and orbits of tuples}
\label{sec:Gcr}

Let $H$ be a subgroup of $G$.  We say that $H$ is {\em $G$-completely reducible} ($G$-cr) if whenever $H$ is contained in an R-parabolic subgroup $P$ of $G$, there is an R-Levi subgroup $L$ of $P$ such that $H$ is contained in $L$.  This notion is due to Serre \cite{serre2}; see \cite{serre1} and \cite{serre1.5} for more details.  In particular, we say that $H$ is {\em $G$-irreducible} ($G$-ir) if $H$ is not contained in any proper R-parabolic subgroup of $G$ at all; then $H$ is $G$-cr.  A $G$-cr subgroup of $G$ is reductive (cf.\ \cite[Sec.~2.5 and Thm.~3.1]{BMR}), and the converse holds in characteristic 0.  A linearly reductive subgroup is $G$-cr, while a nontrivial unipotent subgroup of $G^0$ is never $G$-cr.  A normal subgroup of a $G$-cr subgroup is $G$-cr \cite[Thm.\ 3.10]{BMR}.  We denote by ${\mathcal C}(G)_{\rm cr}$ the set of conjugacy classes of $G$-cr subgroups of $G$.

\begin{lem}
\label{lem:countable}
 ${\mathcal C}(G)_{\rm cr}$ is countable.
\end{lem}

\begin{proof}
 Let $F$ be the algebraic closure of the prime field.  Then $G$ has an $F$-structure, by \cite[Prop.\ 3.2]{Mar}.  By \cite[Thm.\ 10.3]{Mar} and \cite[Thm.\ 3.1]{BMR}, any $G$-cr subgroup of $G$ is $G$-conjugate to an $F$-defined subgroup.  But $G(F)$ has only countably many $G(F)$-conjugacy classes of $G(F)$-cr subgroups since $F$ is countable.  The result follows.
\end{proof}

Let $H$ be a subgroup of $G$.  Let $P=P_\lambda$ be minimal amongst the R-parabolic subgroups of $G$ that contain $H$.  Then $c_\lambda(H)$ is an $L_\lambda$-ir subgroup of $L_\lambda$ (see the proof of \cite[Prop.\ 5.14(i)]{GIT}), so $c_\lambda(H)$ is $G$-cr.  As observed in Section~\ref{sec:intro}, $c_\lambda(H)$ does not depend on the choice of $\lambda$ up to conjugacy, and we set ${\mathcal D}(H)= G\cdot c_\lambda(H)$.  We have ${\mathcal D}(H)= G\cdot H$ if and only if $c_\lambda(H)$ is conjugate to $H$ if and only if $H$ is $G$-cr \cite[Prop.\ 5.14(i)]{GIT}.  For any $\mu\in Y(G)$ such that $H\leq P_\mu$,  if $H$ is $G$-cr then $c_\mu(H)$ is conjugate to $H$, and if $c_\mu(H)$ is $G$-ir then $L_\mu= G$, so $H= c_\mu(H)$ is $G$-ir.  Since $c_\lambda(H)$ is reductive, $R_u(H)\leq R_u(P_\lambda)$ and $H$ is reductive if and only if $H\cap R_u(P_\lambda)$ is finite if and only if ${\rm dim}(H)= {\rm dim}(c_\lambda(H))$.  Moreover, ${\rm dim}(C_G(H))\leq {\rm dim}(C_G(c_\lambda(H)))$, with equality if and only if $H$ is $G$-cr \cite[Thm.\ 5.8(ii)]{GIT}, and ${\rm dim}(c_\lambda(H))= {\rm dim}(H)- {\rm dim}(R_u(H))$.  If $M\leq H$ and $\alpha_H(M)= H/R_u(H)$ then ${\mathcal D}(M)= {\mathcal D}(H)$.

If ${\rm char}(k)= 0$ then $H$ has a Levi subgroup $M$ by \cite[VIII, Thm.\ 4.3]{hochschild}: that is, $H$ has a reductive subgroup $M$ such that $H\cong M\ltimes R_u(H)$.  Then $c_\lambda(H)= c_\lambda(M)$ is conjugate to $M$, since $M$ is $G$-cr, so ${\mathcal D}(H)= G\cdot M$.

The paper \cite{BMR} laid out an approach to the theory of $G$-complete reducibility using geometric invariant theory; we briefly review this now.  As described in Section~\ref{sec:intro}, the idea is to study subgroups of $G$ indirectly by looking instead at generating tuples for subgroups.  Given $s\in {\mathbb N}$ and ${\mathbf g}= (g_1,\ldots, g_s)\in G^s$, we denote by ${\mathcal G}({\mathbf g})$ or ${\mathcal G}( g_1,\ldots, g_s)$ the closed subgroup generated by $g_1,\ldots, g_s$.  If $H$ is of the form ${\mathcal G}(g_1,\ldots, g_s)$ for some $g_1,\ldots, g_s\in G$ then we say that $H$ is {\em topologically finitely generated}, and we call $\tuple{g}$ a {\em generating $s$-tuple} or {\em generating tuple} for $H$.  The structure of the set of generating $s$-tuples is complicated; for instance, if $H= k^*$ and $k$ is solid (Definition~\ref{defn:solid}) then both $\{\tuple{h}\in H^s\mid {\mathcal G}(\tuple{h})= H\}$ and $\{\tuple{h}\in H^s\mid {\mathcal G}(\tuple{h})\neq  H\}$ are dense in $H^s$, even when $s= 1$.

Recall that $G$ acts on $G^N$ by simultaneous conjugation.  We call the quotient space $G^N/G$ a {\em character variety} and we denote the canonical projection from $G^N$ to $G^N/G$ by $\pi_G$.  If $\lambda\in Y(G)$ then we abuse notation and denote the map $c_\lambda\times \cdots\times c_\lambda\colon P_\lambda^N\ra L_\lambda^N$ by $c_\lambda$.  We have $\pi_G({\mathbf g})= \pi_G(c_\lambda({\mathbf g}))$ and ${\mathcal G}(c_\lambda(\tuple{g}))= c_\lambda({\mathcal G}(\tuple{g}))$ for all ${\mathbf g}\in P_\lambda^N$.  If $\tuple{g}\in P_\lambda^N$ and $\tuple{g'}\in P_{\lambda'}^N$ such that $G\cdot c_\lambda(\tuple{g})$ and $G\cdot c_{\lambda'}(\tuple{g'})$ are closed then $\pi_G(\tuple{g})= \pi_G(\tuple{g'})$ if and only if $c_\lambda(\tuple{g})$ and $c_{\lambda'}(\tuple{g'})$ are conjugate (see \cite[Cor.\ 3.5.2]{New}).  In particular, if $G\cdot \tuple{g'}$ is closed then we can take $\lambda'= 0$, so $\pi_G(\tuple{g})= \pi_G(\tuple{g'})$ if and only if $c_\lambda(\tuple{g})$ is conjugate to $\tuple{g'}$.

We need a condition on the field to ensure that reductive groups are topologically finitely generated. 

\begin{defn}
\label{defn:solid}
 An algebraically closed field is {\em solid} if either it has characteristic 0 or it has characteristic $p> 0$ and is transcendental over $\FF_p$.
\end{defn}

The next result allows us to understand subgroups of $G$ by studying generating tuples; several of the results stated above for subgroups have equivalent formulations given for tuples below.

\begin{prop}
\label{prop:topfg}
 \cite[Lem.\ 9.2]{Mar}.  Suppose $k$ is solid.  Let $H$ be a reductive algebraic group and suppose that $N\geq \kappa(H)+ 1$.  Then there exists ${\mathbf h}\in H^N$ such that ${\mathcal G}({\mathbf h})= H$.
\end{prop}

Proposition~\ref{prop:topfg} fails if $k= \ovl{\FF_p}$: for then any topologically finitely generated subgroup of $G$ is finite.  This is the reason for some of the technical complexity in what follows.  We can, however, formulate the results of this section for arbitrary $k$: for example, by using the notion of a ``generic tuple'' \cite[Defn.\ 5.4]{GIT}.  Even when $k$ is solid, non-reductive subgroups need not be topologically finitely generated (for example, a topologically finitely generated subgroup of a unipotent group in positive characteristic is finite).  This is why we need to work with $H/R_u(H)$ rather than just $H$ in Definition~\ref{defn:preceq}.

The next result is \cite[Cor.~3.7]{BMR}.

\begin{thm}
 Let ${\mathbf g}\in G^N$.  Then the orbit $G\cdot {\mathbf g}$ is closed if and only if ${\mathcal G}({\mathbf g})$ is $G$-cr.
\end{thm}

Let $H$ be a reductive subgroup of $G$.  The inclusion of $H^N$ in $G^N$ gives rise to a morphism $\Psi^G_H\colon H^N/H\ra G^N/G$, given by $\Psi^G_H(\pi_H({\mathbf h}))= \pi_G({\mathbf h})$ for ${\mathbf h}\in H^N$.  The next result is \cite[Thm.~1.1]{Mar}.

\begin{thm}
\label{thm:finiteness}
 The morphism $\Psi^G_H$ is finite.  In particular, $\Psi^G_H(H^N/H)$ is closed in $G^N/G$.
\end{thm}

\begin{rem}
\label{rem:genericGir}
 (i) The set $(G^N)_{\rm ir}:= \{{\mathbf g}\in G^N\mid \mbox{${\mathcal G}({\mathbf g})$ is $G$-ir}\}$ is open; this was proved in \cite[Cor.\ 8.4]{Mar} but it also follows from Remark~\ref{rem:Pclosed}.
 
 (ii) Suppose $V$ is irreducible, $N\geq 2$ and there exists $v\in V_0$ such that $G_v^0$ is $G$-ir.  Then $\phi^{-1}((G^N)_{\rm ir})$ is a nonempty open $G$-stable subset of $C$ by (i), and it follows from arguments in Section~\ref{sec:stabcpts} that $\eta(\phi^{-1}((G^N)_{\rm ir}))$ is a dense subset of $V$ (cf.\ Remark~\ref{rem:genericGirstab}).  This means that generic stabilisers are ``large'' in the sense of not being contained in any proper R-parabolic subgroup of $G$.  On the other hand, we can interpret Lemma~\ref{lem:stabdimcty} as saying that generic stabilisers are ``small''.  This special case illustrates the tension between largeness and smallness, from which several of our results spring.
\end{rem}

\section{The partial order $\preceq$}
\label{sec:preceq}

In this section we introduce a technical tool which we need for the proof of Theorem~\ref{thm:genericstab}.  For simplicity, we assume throughout the section that $k$ is solid; see Remark~\ref{rem:nonsolid} for a discussion of arbitrary $k$.

\begin{defn}
\label{defn:preceq}
 Let $H,M$ be subgroups of $G$.  We define $G\cdot H\preceq G\cdot M$ if there exist $s\in {\mathbb N}$, $\tuple{h}\in H^s$ and $\tuple{m}\in M^s$ such that $\alpha_H({\mathcal G}(\tuple{h}))= H/R_u(H)$ and $\pi_G(\tuple{m})= \pi_G(\tuple{h})$.  (It is clear that this does not depend on the choice of subgroup in the conjugacy classes $G\cdot H$ and $G\cdot M$.)  We define $G\cdot H\prec G\cdot M$ if $G\cdot H\preceq G\cdot M$ and $G\cdot H\neq G\cdot M$.
\end{defn}

\begin{lem}
\label{lem:preccrit}
 Let $H,M\leq G$.  Then $G\cdot H\preceq G\cdot M$ if and only if ${\mathcal D}(H)\preceq {\mathcal D}(M)$.
\end{lem}

\begin{proof}
 Pick $\lambda,\mu\in Y(G)$ such that $H\leq P_\lambda$, $c_\lambda(H)$ is $G$-cr, $M\leq P_\mu$ and $c_\mu(M)$ is $G$-cr.  Since ${\mathcal D}(H)= G\cdot c_\lambda(H)$ and ${\mathcal D}(M)= G\cdot c_\mu(M)$, it is enough to show that $G\cdot H\preceq G\cdot M$ if and only if $G\cdot c_\lambda(H)\preceq G\cdot c_\mu(M)$.
 
 So suppose $G\cdot H\preceq G\cdot M$.  There exist $s\in {\mathbb N}$, $\tuple{m}= (m_1,\ldots, m_s)\in M^s$ and $\tuple{h}= (h_1,\ldots, h_s)\in H^s$ such that $\alpha_H({\mathcal G}(\tuple{h}))= H/R_u(H)$ and $\pi_G(\tuple{m})= \pi_G(\tuple{h})$.  Then $c_\mu(\tuple{m})\in c_\mu(M)^s$ and $\pi_G(c_\mu(\tuple{m}))= \pi_G(c_\lambda(\tuple{h}))$.  Now $c_\lambda(H)$ is reductive, so $c_\lambda(R_u(H))= 1$.  It follows that ${\mathcal G}(c_\lambda(\tuple{h}))= c_\lambda({\mathcal G}(\tuple{h}))= c_\lambda(H)$.  This shows that $G\cdot c_\lambda(H)\preceq G\cdot c_\mu(M)$.
 
 Conversely, suppose $G\cdot c_\lambda(H)\preceq G\cdot c_\mu(M)$.  There exist $s\in {\mathbb N}$, $\tuple{y}= (y_1,\ldots, y_s)\in c_\mu(M)^s$ and $\tuple{x}= (x_1,\ldots, x_s)\in c_\lambda(H)^s$ such that ${\mathcal G}(\tuple{x})= c_\lambda(H)$ and $\pi_G(\tuple{y})= \pi_G(\tuple{x})$.  The maps $c_\lambda\colon H^s\ra c_\lambda(H)^s$ and $c_\mu\colon M^s\ra c_\mu(M)^s$ are surjective, so there exist $\tuple{h}= (h_1,\ldots, h_s)\in H^s$ and $\tuple{m}= (m_1,\ldots, m_s)\in M^s$ such that $c_\lambda(\tuple{h})= \tuple{x}$ and $c_\mu(\tuple{m})= \tuple{y}$.
 
 As $c_\lambda(H)$ is reductive, $R_u(H)\leq R_u(P_\lambda)$.  As $(R_u(P_\lambda)\cap H)^0$ is a connected normal unipotent subgroup of $H$, we must have $(R_u(P_\lambda)\cap H)^0\leq R_u(H)$, and it follows that $(R_u(P_\lambda)\cap H)^0= R_u(H)$.  Choose $h_{s+1},\ldots, h_{s+t}\in R_u(P_\lambda)\cap H$ such that the $\alpha_H(h_i)$ for $s+ 1\leq i\leq s+ t$ generate the finite group $(R_u(P_\lambda)\cap H)/R_u(H)$.  Set $\tuple{h}'= (h_1,\ldots, h_s,h_{s+1},\ldots, h_{s+t})\in H^{s+t}$, $\tuple{x}'= (x_1,\ldots, x_s,1,\ldots, 1)\in c_\lambda(H)^{s+t}$, $\tuple{m}'= (m_1,\ldots, m_s,1,\ldots, 1)\in M^{s+t}$ and $\tuple{y}'= (y_1,\ldots, y_s,1,\ldots, 1)\in c_\mu(M)^{s+t}$.  Then $c_\lambda(\tuple{h}')= \tuple{x}'$ and $c_\mu(\tuple{m}')= \tuple{y}'$; moreover, $\alpha_H({\mathcal G}(\tuple{h}'))= H/R_u(H)$ by construction.
 
 To finish, it is enough to show that $\pi_G(\tuple{x}')= \pi_G(\tuple{y}')$.  As $\pi_G(\tuple{x})= \pi_G(\tuple{y})$ and ${\mathcal G}(\tuple{x})= c_\lambda(H)$ is $G$-cr, there exists $\nu\in Y(G)$ such that ${\mathcal G}(\tuple{y})\leq P_\nu$ and $c_\nu(\tuple{y})$ is conjugate to $\tuple{x}$.  It is then immediate that ${\mathcal G}(\tuple{y}')\leq P_\nu$ and $c_\nu(\tuple{y}')$ is conjugate to $\tuple{x}'$.  Hence $\pi_G(\tuple{x}')= \pi_G(\tuple{y}')$, as required.
\end{proof}

\begin{lem}
\label{lem:quotient}
 Let $H,M\leq G$.  Suppose that $H$ is $G$-cr.  Then $G\cdot H\preceq {\mathcal D}(M)$ if and only if $G\cdot H\preceq G\cdot M$ if and only if there exist $\lambda\in Y(G)$ and $M_1\leq P_\lambda\cap M$ such that $c_\lambda(M_1)$ is conjugate to $H$.
\end{lem}

\begin{proof}
 The first equivalence follows from Lemma~\ref{lem:preccrit}.  We prove the second equivalence.  As $H$ is $G$-cr, $H$ is reductive.  Suppose $G\cdot H\preceq G\cdot M$.  There exist $s\in {\mathbb N}$, $\tuple{h}\in H^s$ and $\tuple{m}\in M^s$ such that ${\mathcal G}(\tuple{h})= H$ and $\pi_G(\tuple{m})= \pi_G(\tuple{h})$.  Set $M_1= {\mathcal G}(\tuple{m})$.  As $H= {\mathcal G}(\tuple{h})$ is $G$-cr, there exist $\lambda\in Y(G)$ and $g\in G$ such that $M_1\leq P_\lambda$ and $c_\lambda(\tuple{m})= g\cdot\tuple{h}$.  Then $c_\lambda(M_1)= c_\lambda({\mathcal G}(\tuple{m}))= {\mathcal G}(c_\lambda(\tuple{m}))= {\mathcal G}(g\cdot \tuple{h})= g{\mathcal G}(\tuple{h})g^{-1}= gHg^{-1}$, as required.
 
 Conversely, suppose there exist $\lambda\in Y(G)$ and $M_1\leq P_\lambda\cap M$ such that $c_\lambda(M_1)$ is conjugate to $H$.  Pick $s\geq \kappa(H)+ 1$.  By Proposition~\ref{prop:topfg}, there exists $\tuple{h}\in H^s$ such that ${\mathcal G}(\tuple{h})= H$.  We can pick $\tuple{m}\in M_1^s$ such that $c_\lambda(\tuple{m})$ is conjugate to $\tuple{h}$.  Then $\pi_G(\tuple{m})= \pi_G(c_\lambda(\tuple{m}))= \pi_G(\tuple{h})$, so $G\cdot H\preceq G\cdot M$, and we are done.
\end{proof}

\begin{lem}
\label{lem:transitive}
 Let $H,M,K\leq G$.  If $G\cdot H\preceq G\cdot M$ and $G\cdot M\preceq G\cdot K$ then $G\cdot H\preceq G\cdot K$.
\end{lem}

\begin{proof}
 Suppose $G\cdot H\preceq G\cdot M$ and $G\cdot M\preceq G\cdot K$.  By Lemma~\ref{lem:preccrit}, we can assume $H,M$ and $K$ are $G$-cr.  By Lemma~\ref{lem:quotient}, there exist $\lambda\in Y(G)$ and $K_1\leq P_\lambda\cap K$ such that $c_\lambda(K_1)$ is conjugate to $M$.  Replacing $(K,\lambda)$ with a conjugate of $(K,\lambda)$ if necessary, we can assume that $c_\lambda(K_1)= M$.  Pick $s\in {\mathbb N}$, $\tuple{h}\in H^s$ and $\tuple{m}\in M^s$ such that ${\mathcal G}(\tuple{h})= H$ and $\pi_G(\tuple{h})= \pi_G(\tuple{m})$.  There exists $\tuple{k}\in K_1^s$ such that $c_\lambda(\tuple{k})= \tuple{m}$.  Then $\pi_G(\tuple{k})= \pi_G(c_\lambda(\tuple{k}))= \pi_G(\tuple{m})= \pi_G(\tuple{h})$, so $G\cdot H\preceq G\cdot K$.
\end{proof}

If $H$ and $M$ are subgroups of $G$ and $H$ is conjugate to a subgroup of $M$ then $G\cdot H\preceq G\cdot M$ (and so ${\mathcal D}(H)\preceq {\mathcal D}(M)$ by Lemma~\ref{lem:preccrit}); in particular, $G\cdot H\preceq G\cdot H$.  For without loss we can assume that $H\leq M$, and if we take $s\geq \kappa(H/R_u(H))+ 1$ then by Proposition~\ref{prop:topfg} we can choose $\tuple{m}= \tuple{h}\in H^s$ such that $\alpha_H(\tuple{h})$ generates the reductive group $H/R_u(H)$.  The following example shows that the converse is false, even when $H$ and $M$ are $G$-cr.
 
\begin{exmp}
\label{exmp:not_subgp}
 Let ${\rm char}(k)= 2$, let $G= {\rm SL}_8(k)$ and let $M$ be ${\rm PGL}_3(k)$ embedded in $G$ via the adjoint representation on ${\rm Lie}(M)\cong k^8$.  Since ${\rm Lie}(M)$ is a simple $M$-module, $M$ is $G$-cr (in fact, $G$-ir).  It follows from elementary representation-theoretic arguments that $M$ contains exactly two subgroups of type $A_1$ up to $M$-conjugacy: the derived group $H_1$ of a Levi subgroup of a rank 1 parabolic subgroup of $M$, and the image $H_2$ of ${\rm SL}_2(k)$ under the map ${\rm SL}_2(k)\ra {\rm SL}_3(k)\ra M$, where the first arrow is the adjoint representation of ${\rm SL}_2(k)$ and the second is the canonical projection.  It is easily checked that $H_1$ is $M$-cr but $H_2$ is not: in fact, there exists $\lambda\in Y(M)$ such that $c_\lambda(H_2)= H_1$.
 
 Now $H_1$ is not $G$-cr because ${\rm Lie}(H_1)$ is an $H_1$-stable submodule of ${\rm Lie}(M)$ and $H_1$ does not act completely reducibly on ${\rm Lie}(H_1)$.  Choose $\mu\in Y(G)$ such that $H_1\leq P_\mu$ and $H:= c_\mu(H_1)$ is $G$-cr.  We have $G\cdot H_1\preceq G\cdot M$ as $H_1\leq M$, so $G\cdot H\preceq G\cdot M$ by Lemma~\ref{lem:preccrit}.  We claim that $H$ is not $G$-conjugate to a subgroup of $M$.  First, $H$ is not $G$-conjugate to $H_1$ because $H$ is $G$-cr but $H_1$ is not.  If $H$ is $G$-conjugate to $H_2$ then $H_2$ is $G$-cr, so $H_1= c_\lambda(H_2)$ is $G$-conjugate to $H_2$; but then $H$ is $G$-conjugate to $H_1$, a contradiction.  This proves the claim.
\end{exmp}
 
We do, however, have the following result.

\begin{lem}
\label{lem:antisymm}
 Let $H,M\leq G$.  If $G\cdot H\preceq G\cdot M$ and $G\cdot M\preceq G\cdot H$ then ${\mathcal D}(H)= {\mathcal D}(M)$.  In particular, if $H$ and $M$ are $G$-cr then $G\cdot H= G\cdot M$.
\end{lem}

\begin{proof}
 By Lemma~\ref{lem:preccrit}, we can assume $H$ and $M$ are $G$-cr; in particular, $H$ and $M$ are reductive.  By Lemma~\ref{lem:quotient}, there exist $\lambda\in Y(G)$ and $M_1\leq P_\lambda\cap M$ such that $c_\lambda(M_1)$ is conjugate to $H$.  Replacing $(M,\lambda)$ with a conjugate of $(M,\lambda)$ if necessary, we can assume that $c_\lambda(M_1)= H$.  We have
 $$ {\rm dim}(H)= {\rm dim}(c_\lambda(M_1))\leq {\rm dim}(M_1)\leq {\rm dim}(M). $$
 By symmetry, ${\rm dim}(M)\leq {\rm dim}(H)$, so
 $$ {\rm dim}(H)= {\rm dim}(c_\lambda(M_1))= {\rm dim}(M_1)= {\rm dim}(M). $$
 It now follows that
  $$ \kappa(H)= \kappa(c_\lambda(M_1))\leq \kappa(M_1)\leq \kappa(M). $$
  By symmetry, $\kappa(M)\leq \kappa(H)$, so
  $$ \kappa(H)= \kappa(c_\lambda(M_1))= \kappa(M_1)= \kappa(M). $$
  This implies that $M_1= M$ since $M_1\leq M$, so $H= c_\lambda(M)$.  But $M$ is $G$-cr, so $M$ is conjugate to $H$.  This completes the proof.
\end{proof}

The next result follows immediately from Lemmas~\ref{lem:transitive} and \ref{lem:antisymm}.

\begin{cor}
\label{cor:po}
 The relation $\preceq$ is a partial order on ${\mathcal C}(G)_{\rm cr}$.
\end{cor}

\begin{rem}
\label{rem:dcc}
 The proof of Lemma~\ref{lem:antisymm} shows that if $H$ and $M$ are $G$-cr subgroups of $G$ and $G\cdot H\prec G\cdot M$ then either ${\rm dim}(H)< {\rm dim}(M)$, or ${\rm dim}(H)= {\rm dim}(M)$ and $\kappa(H)< \kappa(M)$.  It follows that ${\mathcal C}(G)_{\rm cr}$ satisfies the descending chain condition with respect to $\preceq$. 
\end{rem}

Given a reductive subgroup $H$ of $G$, set $S(H)= \{\tuple{g}\in G^N\mid \pi_G(\tuple{g})\in \Psi^G_H(H^N/H)\}$.  Theorem~\ref{thm:finiteness} implies that $S(H)$ is closed.

\begin{lem}
\label{lem:Psiimage}
 Let $\tuple{g}\in G^N$ and let $H\leq G$ be reductive.  Then $\tuple{g}\in S(H)$ if and only if $G\cdot {\mathcal G}(\tuple{g})\preceq G\cdot H$ if and only if ${\mathcal D}({\mathcal G}(\tuple{g}))\preceq {\mathcal D}(H)$.
\end{lem}

\begin{proof}
 We prove the first equivalence.  If $\tuple{g}\in S(H)$ then there exists $\tuple{h}\in H^N$ such that $\pi_G(\tuple{h})= \pi_G(\tuple{g})$, so $G\cdot {\mathcal G}(\tuple{g})\preceq G\cdot H$ as $\tuple{g}$ generates ${\mathcal G}(\tuple{g})$.  Conversely, suppose $G\cdot {\mathcal G}(\tuple{g})\preceq G\cdot H$.  Set $M= {\mathcal G}(\tuple{g})$.  Then ${\mathcal D}(M)\preceq {\mathcal D}(H)$ by Lemma~\ref{lem:preccrit}.  Choose $\mu\in Y(G)$ such that $H\leq P_\mu$ and $c_\mu(H)$ is $G$-cr.    Choose $\nu\in Y(G)$ such that $M\leq P_\nu$ and $c_\nu(M)$ is $G$-cr.  Then ${\mathcal D}(H)= G\cdot c_\mu(H)$ and ${\mathcal D}(M)= G\cdot c_\nu(M)$.  By Lemma~\ref{lem:quotient}, there exist $K\leq c_\mu(H)$ and $\lambda\in Y(G)$ such that $G\cdot c_\lambda(K)= G\cdot c_\nu(M)$.  There exists $\tuple{k}\in K^N$ such that $G\cdot c_\lambda(\tuple{k})= G\cdot c_\nu(\tuple{g})$.  There exists $\tuple{h}\in H^N$ such that $c_\mu(\tuple{h})= \tuple{k}$.  We have $\pi_G(\tuple{h})= \pi_G(c_\mu(\tuple{h}))= \pi_G(\tuple{k})= \pi_G(c_\lambda(\tuple{k}))= \pi_G(c_\nu(\tuple{g}))= \pi_G(\tuple{g})$, so $\tuple{g}\in S(H)$, as required.
 
 The second equivalence follows from Lemma~\ref{lem:preccrit}.
\end{proof}

To prove our results in Section~\ref{sec:genericstab}, we need to investigate the behaviour of the relation $\preceq$ under field extensions.  We assume for the rest of the section that $N\geq \Theta+ 1$.  Fix a $G$-cr subgroup $H$ of $G$ such that $N\geq \kappa(H)+ 1$.
 
\begin{defn}
 Define $B_H= \{v\in V\mid {\mathcal D}(G_v)= G\cdot H\}$.
\end{defn}
 
 Let $v\in V$.  For all $\tuple{g}\in G_v^N$, we have ${\mathcal G}(\tuple{g})\leq G_v$, and so ${\mathcal D}({\mathcal G}(\tuple{g}))\preceq {\mathcal D}(G_v)$.  Moreover, since  $N\geq \Theta+ 1\geq \kappa(G_v)+1\geq \kappa(G_v/R_u(G_v))+1$, there exists $\tuple{g'}\in G^N$ such that $\alpha_{G_v}({\mathcal G}(\tuple{g'}))= G_v/R_u(G_v)$ by Proposition~\ref{prop:topfg}, so ${\mathcal D}({\mathcal G}(\tuple{g'}))= {\mathcal D}(G_v)$.  Lemma~\ref{lem:transitive} now implies that ${\mathcal D}(G_v)\preceq G\cdot H$ if and only if ${\mathcal D}({\mathcal G}(\tuple{g}))\preceq {\mathcal D}(H)$ for all $\tuple{g}\in G_v^N$ if and only if $\pi_G(\tuple{g})\in S(H)$ for all $\tuple{g}\in G_v^N$, where the last equivalence follows from Lemma~\ref{lem:Psiimage}.  This is the case if and only if the following formula holds:
 \begin{equation}
 \label{eqn:leq}
  (\forall \tuple{g}\in G_v^N) \ (\exists  \tuple{h}\in H^N) \ \pi_G(\tuple{h})= \pi_G(\tuple{g}).
 \end{equation}
Conversely, $G\cdot H\preceq {\mathcal D}(G_v)$ if and only if there exist $M_1\leq G_v$ and $\lambda\in Y(G)$ such that $M_1\leq P_\lambda$ and $c_\lambda(M_1)$ is conjugate to $H$ (Lemma~\ref{lem:quotient}).  This is the case if and only if the following formula holds:
 \begin{equation}
 \label{eqn:geq}
  (\exists \tuple{g}\in G_v^N) \ (\exists g\in G) \ \pi_G(\tuple{g})= g\cdot \tuple{h_0},
 \end{equation}
where $\tuple{h_0}$ is a fixed element of $H^N$ such that ${\mathcal G}(\tuple{h}_0)= H$.  For, given $ \tuple{g}\in G_v^N$ and $g\in G$ such that $\pi_G(\tuple{g})= g\cdot \tuple{h_0}$, we set $M_1= {\mathcal G}(\tuple{g})$; conversely, given $M_1\leq G_v$ and $\lambda\in Y(G)$ such that $M_1\leq P_\lambda$ and $g\in G$ such that $c_\lambda(M_1)= gHg^{-1}$, we choose $\tuple{g}\in M_1^N$ such that $c_\lambda(\tuple{g})= g\cdot \tuple{h_0}$.

We summarise the above argument as follows.

\begin{lem}
\label{lem:constructible}
 Let $H$ be a $G$-cr subgroup of $G$ such that $N\geq \kappa(H)+ 1$.  Then $B_H\subseteq V$ is the set of solutions to the formulas Eqn.\ (\ref{eqn:leq}) and Eqn.\ (\ref{eqn:geq}).  In particular, $B_H$ is constructible.
\end{lem} 

\begin{rem}
\label{rem:nonsolid}
 It can be shown that Lemma~\ref{lem:constructible} holds for arbitrary $k$, where we take $\tuple{h}$ to be a generic tuple for $H$ in the sense of \cite[Defn.\ 5.4]{GIT}.  To do this, one replaces generating tuples with generic tuples in the definition of $\preceq$ and makes the obvious modifications to the arguments of this section.
\end{rem}

\section{Proof of Theorem~\ref{thm:genericstab}}
\label{sec:genericstab}

We assume throughout the section that $N\geq \Theta+ 1$.

\begin{proof}[Proof of Theorem~\ref{thm:genericstab}]
 We will show that there is a $G$-cr subgroup $H$ of $G$ such that $N\geq \kappa(H)+ 1$ and $B_H$ has nonempty interior.  By Lemma~\ref{lem:constructible} and Remark~\ref{rem:nonsolid}, it is enough to prove this after extending the ground field to an uncountable algebraically closed field $\Omega$ (recall from the proof of Lemma~\ref{lem:countable} that any $G(\Omega)$-cr subgroup of $G(\Omega)$ is $G(\Omega)$-conjugate to a $k$-defined $G$-cr subgroup).  Thus we can assume without loss that $k$ is uncountable (and hence solid).
 
 Let $D_1,\ldots, D_t$ be the irreducible components of $C$ such that $\ovl{\eta(G\cdot D_j)}= V$ for $1\leq j\leq t$---it follows from Lemma~\ref{lem:cpts}(b) below that there is at least one such component---and let $D_1',\ldots, D_{t'}'$ be the other irreducible components of $C$.  Let $V'= V\backslash \bigcup_{j=1}^{t'} \ovl{\eta(G\cdot D_{j}')}$.  For $1\leq j\leq t$, set $E_j= \{(v,\tuple{g})\in D_j\mid \alpha_{G_v}({\mathcal G}(\tuple{g}))= G_v/R_u(G_v)\}$; note that $E_j$ is neither closed nor open in general, and if $(v,\tuple{g})\in E_j$ then ${\mathcal D}({\mathcal G}(\tuple{g}))= {\mathcal D}(G_v)$.  For any $v\in V'$, $N\geq \Theta+ 1\geq \kappa(G_v)+1\geq \kappa(G_v/R_u(G_v))+1$, so by Proposition~\ref{prop:topfg} there exists $\tuple{g}\in G^N$ such that $\alpha_{G_v}({\mathcal G}(\tuple{g}))= G_v/R_u(G_v)$.  Then $(v,\tuple{g})\in D_j$ for some $1\leq j\leq t$, so $(v,\tuple{g})\in E_j$.  Hence $\bigcup_{1\leq j\leq t} \eta(G\cdot E_j)\supseteq V'$.  As $G$ permutes the irreducible components of $V$ transitively, $\eta(G\cdot E_m)$ is dense in $V$ for some $1\leq m\leq t$.
 
 Choose $G$-cr subgroups $H_i$ such that ${\mathcal H}:= \{H_i\mid i\in I\}$ is a set of representatives for the conjugacy classes in ${\mathcal C}(G)_{\rm cr}$; by Lemma~\ref{lem:countable}, $I$ is countable.  Let $\Lambda= \{H_i\mid G\cdot D_m\subseteq \phi^{-1}(S(H_i))\}$.  Then $G\in \Lambda$, so $\Lambda$ is nonempty.  By Remark~\ref{rem:dcc}, we can pick $H\in \Lambda$ such that $H$ is minimal with respect to $\preceq$.   We claim that $G\cdot D_j\subseteq \phi^{-1}(S(H))$ for all $1\leq j\leq t$.  To prove this, let $(v,\tuple{g})\in D_j$ such that $v\in \eta(E_m)$.  There exists $\tuple{g}'\in G_v^N$ such that $(v,\tuple{g}')\in E_m$.  Then $(v,\tuple{g}')\in \phi^{-1}(S(H))$, so $\tuple{g}'\in S(H)$.  Now ${\mathcal G}(\tuple{g})\leq G_v$, so ${\mathcal D}({\mathcal G}(\tuple{g}))\preceq {\mathcal D}(G_v)= {\mathcal D}({\mathcal G}(\tuple{g}'))\preceq G\cdot H$ by Lemma~\ref{lem:Psiimage}.  Hence $(v,\tuple{g})\in \phi^{-1}(S(H))$ by Lemma~\ref{lem:Psiimage}.  As $S(H)$ is $G$-stable, it now follows that if $(v,\tuple{g})\in D_j$ and $v\in \eta(G\cdot E_m)$ then $(v,\tuple{g})\in \phi^{-1}(S(H))$.  But $\eta^{-1}(\eta(G\cdot E_m))\cap D_j$ is dense in $D_j$ as $\eta(G\cdot E_m)$ is dense in $V$, so $D_j\subseteq \phi^{-1}(S(H))$.  As $S(H)$ is $G$-stable, $G\cdot D_j\subseteq \phi^{-1}(S(H))$, as claimed.  It follows from Lemma~\ref{lem:Psiimage} that ${\mathcal D}({\mathcal G}(\tuple{g}))\preceq G\cdot H$ for all $1\leq j\leq t$ and all $(v,\tuple{g})\in G\cdot D_j$.  In particular, for any $v\in V'$, there exist $j$ and $\tuple{g'}\in G^N$ such that $(v,\tuple{g'})\in E_j$, so ${\mathcal D}(G_v)= {\mathcal D}(\tuple{g'})\preceq G\cdot H$.
 
 To finish, we show that $B_H$ has nonempty interior in $V$.  Suppose otherwise.  As $B_H$ is constructible (Lemma~\ref{lem:constructible}), $\ovl{B_H}$ is a proper closed subset of $V$, so $V\backslash B_H$ is a $G$-stable subset with nonempty interior.  Now $\eta(\phi^{-1}(S(H))$ is dense in $V$ as it contains $\eta(G\cdot D_m)$.  Hence there is a nonempty open $G$-stable subset $O$ of $\eta(\phi^{-1}(S(H)))\cap V'$ such that $B_H\cap O$ is empty.  Let $v\in O$ and let $\tuple{g}\in G^N$ such that $(v,\tuple{g})\in D_m$.  Then ${\mathcal D}({\mathcal G}(\tuple{g}))\preceq {\mathcal D}(G_v)\preceq G\cdot H$; but $v\not\in B_H$, so ${\mathcal D}(G_v)\neq G\cdot H$, and it follows from Corollary~\ref{cor:po} that ${\mathcal D}({\mathcal G}(\tuple{g}))\prec G\cdot H$.  Hence ${\mathcal D}({\mathcal G}(\tuple{g}))= G\cdot H_i$ for some $i\in I$ such that $G\cdot H_i\prec G\cdot H$.  Lemma~\ref{lem:Psiimage} now implies that $\eta^{-1}(O)\cap D_m\subseteq \bigcup_{i\in I'} \phi^{-1}(S(H_i))$, where $I':= \{i\in I\mid G\cdot H_i\prec G\cdot H\}$.  By Corollary~\ref{cor:qcmpctirred}, there exists $i\in I'$ such that $\eta^{-1}(O)\cap D_m\subseteq \phi^{-1}(S(H_i))$.  Since $\eta^{-1}(O)\cap D_m$ is a nonempty open subset of $D_m$ and $\phi^{-1}(S(H_i))$ is closed and $G$-stable, $G\cdot D_m\subseteq \phi^{-1}(S(H_i))$.  But $G\cdot H_i\prec G\cdot H$, which contradicts the minimality of $H$.  We conclude that $B_H$ has nonempty interior in $V$ after all.  Finally, since $G\cdot H= {\mathcal D}(G_v)$ for some $v\in V$, we have $\kappa(H)\leq \kappa(G_v)\leq \Theta$, so $N\geq \kappa(H)+ 1$.  This completes the proof. 
\end{proof}

\begin{proof}[Proof of Corollaries~\ref{cor:Gcrprinc} and \ref{cor:char0princ}]
 We can assume $O$ is $G$-stable.  By Theorem~\ref{thm:genericstab}, there is a nonempty open $G$-stable subset $O'$ of $V$ and a $G$-cr subgroup $H$ of $G$ such that ${\mathcal D}(G_v)= G\cdot H$ for all $v\in O'$.  Now $O\cap O'$ is a nonempty open $G$-stable subset of $V$, and for all $v\in O\cap O'$, ${\mathcal D}(G_v)= G\cdot H$.  Since $G_v$ is $G$-cr for $v\in O\cap O'$, $G_v$ is conjugate to $H$.  It follows that $V$ has a principal stabiliser.
 
 In particular, the hypotheses of Corollary~\ref{cor:Gcrprinc} are satisfied if ${\rm char}(k)= 0$ and $V_{\rm red}$ is nonempty, since then $V_{\rm red}$ is open by Theorem~\ref{thm:main} and for all $v\in V_{\rm red}$, $G_v$---being reductive---is $G$-cr.  This proves Corollary~\ref{cor:char0princ}.
\end{proof}

\begin{rem}
\label{rem:Leviconj}
 Here is a generalisation of Corollary~\ref{cor:char0princ}.  If ${\rm char}(k)= 0$ and $O$ is as in Theorem~\ref{thm:genericstab} then $G\cdot M_v= {\mathcal D}(G_v)= G\cdot H$ for all $v\in O$, where $M_v$ is any Levi subgroup of $G_v$.
\end{rem}

\section{Irreducible components of the stabiliser variety}
\label{sec:stabcpts}

In this section we study the irreducible components of the stabiliser variety $C$.  We use the information we obtain to prove results analogous to those in Section~\ref{sec:genericstab}, but for the subgroups $G_v^0$ rather than the subgroups $G_v$.  We assume throughout the section that $N\geq 3$.

\begin{lem}
\label{lem:cpts}
 (a) Let $D$ be an irreducible component of $C$ such that $\eta(G\cdot D)$ is dense in $V$.  Then ${\rm dim}(D)= n+ Nr$ and for all $v\in V_0$, the fibre $\left(\eta|_D\right)^{-1}(v)$ either is empty or has dimension $Nr$ and is isomorphic (via $\phi$) to a union of irreducible components of $G_v^N$.
 \smallskip\\
 (b) There is a unique closed subset $\widetilde{C}$ of $C$ such that $\widetilde{C}$ contains $V\times \{{\mathbf 1}\}$, $\widetilde{C}$ is a union of irreducible components of $C$ and $G$ permutes these irreducible components transitively.  The variety $\widetilde{C}$ is the closure of the set $\{(v,{\mathbf g})\mid v\in V_0, {\mathbf g}\in (G_v^0)^N\}$, and each irreducible component of $\widetilde{C}$ has dimension $n+ Nr$.
\end{lem}

\begin{proof}
 Clearly it is enough to prove the result when $G$ is connected and $V$ is irreducible, so we assume this.\smallskip
 
 \noindent (a)
 Define $f\colon V\times G^N \ra V\times V^N$ by
$$ f(v,{\mathbf g})= (v,g_1\cdot v,\ldots, g_N\cdot v). $$
Let $Y$ be the closure of the image of $f$.  Let $\Delta$ be the diagonal in $V\times V^N$; then $C= f^{-1}(\Delta)$.  The variety $Y$ is irreducible because $G$ and $V$ are irreducible.  Let $v\in V$ and let ${\mathbf g}\in G^N$.  Then $f^{-1}(v,g_1\cdot v,\ldots, g_N\cdot v)= \{v\}\times g_1G_v,\ldots, g_NG_v$.  Hence irreducible components of generic fibres of $f$ over $Y$ have dimension $Nr$.  It follows that ${\rm dim}(Y)= {\rm dim}(V\times G^N)- Nr = n+ N{\rm dim}(G)- Nr= n+ N({\rm dim}(G)- r)$.  Since $\eta(D)$ is dense in $V$, $f(D)$ is dense in $\Delta$.  Hence ${\rm dim}(D)\geq {\rm dim}(\Delta)+ Nr= n+ Nr$.

If $v\in \eta(D)\cap V_0$ and $Z$ is an irreducible component of $\left(\eta|_D\right)^{-1}(v)$ then ${\rm dim}(Z)\geq {\rm dim}(D)- {\rm dim}(V)\geq n+ Nr- n= Nr$.  But $\phi(\eta^{-1}(v))$ is a subset of $G_v^N$ and the irreducible components of $G_v^N$ all have dimension ${\rm dim}(G_v^N)= Nr$.  This forces $Z$ to be isomorphic (via $\phi$) to an irreducible component of $G_v^N$.  Hence irreducible components of generic fibres of $\eta|_D$ have dimension $Nr$, which implies that ${\rm dim}(D)= n+ Nr$.  Part (a) now follows.
\smallskip\\
 (b) Since $V\times \{{\mathbf 1}\}$ is irreducible, there is some irreducible component $\widetilde{C}$ of $C$ such that $\widetilde{C}$ contains $V\times \{{\mathbf 1}\}$.  For any $v\in V_0$, let $Z$ be an irreducible component of the fibre $\left(\eta|_{\widetilde{C}}\right)^{-1}(v)$ such that $(v,\tuple{1})\in Z$.  By part (a), ${\rm dim}(Z)= Nr$, so $Z$ is isomorphic via $\phi$ to an irreducible component of $G_v^N$.  But the only component of $G_v^N$ that contains ${\mathbf 1}$ is $(G_v^0)^N$, so $\{v\}\times (G_v^0)^N\subseteq Z$.  Hence $\widetilde{C}$ contains the closure of $\{(v,{\mathbf g})\mid v\in V_0, {\mathbf g}\in (G_v^0)^N\}$---call this closure $C'$.
 
 Let $A_1,\ldots, A_m$ be the irreducible components of $C'$ such that $\ovl{\eta(A_j)}= V$ (there is at least one, since $\eta(C')= V$).  Let $s_i= {\rm dim}(A_i)$ for $1\leq i\leq m$ and let $\eta_i\colon A_i\ra V$ be the restriction of $\eta$.  There is a nonempty open subset $U$ of $V$ such that for all $v\in U$, $\eta^{-1}(v)\subseteq A_1\cup\cdots \cup A_m$ and every irreducible component of $\eta_i^{-1}(v)$ has dimension $s_i- n$.  Since $\{v\}\times (G_v^0)^N\subseteq C'$ for all $v\in V_0$, if $v\in U\cap V_0$ then $\eta_j^{-1}(v)$ must contain $\{v\}\times (G_v^0)^N$ for some $1\leq j\leq m$, which forces $s_j\geq n+ Nr$.  But ${\rm dim}(\widetilde{C})= n+ Nr$ by part (a), so $A_j$ must be the whole of $\widetilde{C}$, so $C'= \widetilde{C}$.  This completes the proof.
\end{proof}

\begin{rem}
\label{rem:pathologies}
 The dimension inequality in Lemma~\ref{lem:cpts}(a) can fail if $\eta(G\cdot D)$ is not dense in $V$ (Example~\ref{exmp:unipotent}).  Moreover, $\widetilde{C}$ need not contain the whole of $\{(v,{\mathbf g})\mid v\in V, {\mathbf g}\in (G_v^0)^N\}$: see Examples~\ref{exmp:cosets}(a) and \ref{exmp:unipotent}.
\end{rem}

\begin{rem}
 If $G$ is connected and $V$ is irreducible then $\widetilde{C}$ is irreducible and $G$-stable.  More generally, any irreducible component of $C$ is $G$-stable in this case.
\end{rem}

\begin{defn}
 We call $\widetilde{C}$ the {\em connected-stabiliser variety} of $V$.
\end{defn}

\begin{cor}
\label{cor:finstab}
 If $r= 0$ then $\widetilde{C}= V\times \{\tuple{1}\}$.
\end{cor}

\begin{proof}
 The irreducible components of $V\times \{\tuple{1}\}$ are isomorphic via $\eta$ to the irreducible components of $V$, so they are permuted transitively by $G$ and each has dimension $n$.  It follows from the dimension formula in Lemma~\ref{lem:cpts}(a) that these irreducible components are irreducible components of $C$.  The result now follows from Lemma~\ref{lem:cpts}(b).
\end{proof}

We denote by $\widetilde{\phi}\colon \widetilde{C}\ra G^N$ and $\widetilde{\eta}\colon \widetilde{C}\ra V$ the restrictions to $\widetilde{C}$ of $\phi$ and $\eta$, respectively, and if $v\in V$ then we denote $\widetilde{\phi}(\widetilde{\eta}^{-1}(v))$ by $F_v$.  If $v\in V_0$ then $(G_v^0)^N\subseteq F_v$; we do not know whether equality holds for all $v\in V_0$, or even for generic $v\in V_0$.

We now give a counterpart to Theorem~\ref{thm:genericstab}.  In the connected case, we obtain slightly more information: we can describe ${\mathcal D}(G_v^0)$ for all $v\in V_{\rm min}$  (recall the definition of $V_{\rm min}$ from Remark~\ref{rem:Vmin}).

\begin{thm}
\label{thm:genericstab_conn}
 There exists a connected $G$-completely reducible subgroup $H$ of $G$  such that:
 \begin{itemize}
  \item[(a)] for all $v\in V_{\rm min}$, ${\mathcal D}(G_v^0)= G\cdot H$.
  \item[(b)] $\widetilde{C}\subseteq \widetilde{\phi}^{-1}(S(H))$.
 \end{itemize}
 In particular, if $V_{\rm red}$ is nonempty then ${\mathcal D}(G_v^0)= G\cdot H$ for all $v\in V_{\rm red}$.
\end{thm}
  
\begin{proof}
 By Theorem~\ref{thm:genericstab}, there exist a $G$-cr subgroup $H'$ of $G$ and a $G$-stable open subset $O$ of $V$ such that ${\mathcal D}(G_v)= G\cdot H'$ for all $v\in O$.  Set $H= (H')^0$; then $H$ is $G$-cr as $H\unlhd H'$.   Let $t$ be the minimal dimension of ${\rm dim}(R_u(G_v))$ for $v\in V_0$.  The $G$-stable open sets $O$ and $V_{\rm min}$ have nonempty intersection, so there exists $v\in V_{\rm min}\cap O$ such that ${\mathcal D}(G_v)= G\cdot H'$ and ${\rm dim}(R_u(G_v))= t$.  This yields ${\rm dim}(H')= {\rm dim}(G_v)- {\rm dim}(R_u(G_v))= r- t$.
 
  By the proof of Theorem~\ref{thm:genericstab}, $\widetilde{C}\subseteq \widetilde{\phi}^{-1}(S(H'))$.  Let $v\in V_{\rm min}$ and choose $\lambda\in Y(G)$ such that ${\mathcal D}(G_v)= c_\lambda(G_v)$.  Then $c_\lambda(G_v)$ is $G$-cr and $c_\lambda(G_v^0)= c_\lambda(G_v)^0$ is a normal subgroup of $c_\lambda(G_v)$, so $c_\lambda(G_v^0)$ is $G$-cr.  It follows that ${\mathcal D}(G_v^0)= G\cdot c_\lambda(G_v^0)$.  We want to prove that $c_\lambda(G_v^0)$ is conjugate to $H$: that is, we want to prove that
 \begin{equation}
 \label{eqn:conj_conn}
  (\exists m\in G) \ [(\forall g\in G_v^0) \ c_\lambda(g)\in mHm^{-1} \wedge (\forall h\in H) \ (\exists g\in G_v^0) \ c_\lambda(g)= mhm^{-1}].
 \end{equation}
 Since (\ref{eqn:conj_conn}) is a first-order formula, this is a constructible condition.  Hence it is enough to prove that it holds after extending $k$ to any larger algebraically closed field.  So without loss of generality we assume $k$ is solid.
 
 By Proposition~\ref{prop:topfg}, we can choose $\tuple{g'}\in (G_v^0)^N$ such that $\alpha_{G_v^0}({\mathcal G}(\tuple{g'}))= G_v^0/R_u(G_v^0)$.  There exists $\tuple{h}\in (H')^N$ such that $\pi_G(\tuple{h})= \pi_G(\tuple{g'})$.  Let $K= {\mathcal G}(\tuple{h})$.  Now $c_\lambda({\mathcal G}(\tuple{g'}))= c_\lambda(G_v^0)$ is $G$-cr, so there exists $\mu\in Y(G)$ such that $c_\mu(\tuple{h})$ is conjugate to $c_\lambda(\tuple{g'})$.  Then $c_\lambda({\mathcal G}(\tuple{g'}))$ is conjugate to $c_\mu({\mathcal G}(\tuple{h}))$.  But ${\rm dim}(c_\lambda({\mathcal G}(\tuple{g'})))= {\rm dim}(c_\lambda(G_v^0))= {\rm dim}(H)\geq {\rm dim}(K)\geq {\rm dim}(c_\mu(K))= {\rm dim}(c_\mu({\mathcal G}(\tuple{h})))$, which forces ${\rm dim}(K)$ to equal ${\rm dim}(H)$.  Hence $K\supseteq H$.  Now $c_\mu(K)$ is conjugate to $c_\lambda(G_v^0)$, which is connected, so $c_\mu(K)= c_\mu(H)$.  But $c_\mu(H)$ is conjugate to $H$ since $H$ is $G$-cr, so we deduce that $c_\lambda(G_v^0)$ is conjugate to $H$.  Hence ${\mathcal D}(G_v^0)= G\cdot H$.  This proves part (a).  Moreover, if $\tuple{g}\in (G_v^0)^N$ then $c_\lambda({\mathcal G}(\tuple{g}))= {\mathcal G}(c_\lambda(\tuple{g}))$ is conjugate to a subgroup of $H$, so there exists $\tuple{h}\in H^N$ such that $c_\lambda(\tuple{g})$ is conjugate to $\tuple{h}$; hence $(v,\tuple{g})\in \widetilde{\phi}^{-1}(S(H))$.  As $\{(v,\tuple{g})\mid v\in V_{\rm min}, \tuple{g}\in (G_v^0)^N\}$ is dense in $\widetilde{C}$ by Lemma~\ref{lem:cpts}(b) and Remark~\ref{rem:Vmin}, $\widetilde{C}\subseteq \widetilde{\phi}^{-1}(S(H))$.  This proves part (b).
\end{proof}

The next result is the counterpart to Corollaries~\ref{cor:Gcrprinc} and \ref{cor:char0princ}.  We omit the proof, which is similar.

\begin{cor}
\label{cor:Gcrprinc_conn}
 Suppose there is a nonempty open subset $O$ of $V$ such that $G_v^0$ is $G$-cr for all $v\in O$ (in particular, this condition holds if ${\rm char}(k)= 0$ and $V_{\rm red}$ is nonempty)).  Let $H$ be the connected $G$-cr subgroup from Theorem~\ref{thm:genericstab_conn}.  Then $G_v^0$ is conjugate to $H$ for all $v\in V_{\rm red}$.
\end{cor}

\begin{rem}
\label{rem:genericGirstab}
 Suppose there exists $v\in V_0$ such that $G_v^0$ is $G$-ir.  Then $G_v$ is $G$-cr, so $v\in V_{\rm red}$, so $V_{\rm red}$ is nonempty.  We have ${\mathcal D}(G_v^0)= G\cdot H$ by Theorem~\ref{thm:genericstab_conn}(a).  As $G_v^0$ is $G$-cr, $G\cdot G_v^0= G\cdot H$.  It follows that $H$ is $G$-ir and $G_w^0$ is conjugate to $H$ for all $w\in V_{\rm red}$; in particular, $G_w^0$ is $G$-ir for all $w\in V_{\rm red}$.  The analogous result for the full stabiliser $G_v$ is false (cf.\ Remarks~\ref{rem:notgeneric} and ~\ref{rem:genericGir}(ii), and Examples~\ref{exmp:cosets}(c) and \ref{exmp:unipotent}).  However, if $O$ is as in Theorem~\ref{thm:genericstab} and there exists $v\in O$ such that $G_v$ is $G$-ir then an argument like the one above shows that $V$ has a $G$-ir principal stabiliser.
\end{rem}

Theorem~\ref{thm:genericstab_conn} gives rise to the following counterpart to Remark~\ref{rem:Leviconj} for $G_v^0$; the proof is similar.

\begin{cor}  
 Suppose ${\rm char}(k)= 0$.  Then $H$ is conjugate to a Levi subgroup of $G_v^0$ for all $v\in V_{\rm min}$.
\end{cor}

We give a criterion to ensure that the fibres of $\widetilde{\eta}$ are irreducible.  Define $\widetilde{C}_{\rm min}= \widetilde{\eta}^{-1}(V_{\rm min})$.

\begin{prop}
 Suppose ${\rm char}(k)= 0$ and $N\geq \Theta+ 1$.  Then $\widetilde{C}_{\rm min}= \{(v,\tuple{g})\mid v\in V_{\rm min}, \tuple{g}\in (G_v^0)^N\}$.
\end{prop}

\begin{proof}
 Let $H$ be as in Theorem~\ref{thm:genericstab_conn}.  Let $v\in V_{\rm min}$, and suppose $F_v$ properly contains $(G_v^0)^N$.  Then $F_v$ contains an irreducible component $D\neq (G_v^0)^N$ of $G_v^N$ by Lemma~\ref{lem:cpts}(a).  Set $K= G_v$, set $M= K/R_u(K)$ and let $\alpha_K\colon K\ra M$ be the canonical projection.  Let $K_1$ be the subgroup of $K$ generated by $K^0$ together with the components of each of the tuples in $D$; then $K_1$ properly contains $K^0$.  As $R_u(G_v)$ is connected, $M_1:= \alpha_K(K_1)$ properly contains $M^0$.  In particular, $M_1$ is reductive.  By \cite[Lem.\ 9.2]{Mar}, there exists $\tuple{g}\in D$ such that $\alpha_K(\tuple{g})$ generates $M_1$.  Hence $G\cdot K_1\preceq G\cdot {\mathcal G}(\tuple{g})$.  Now $\tuple{g}\in \widetilde{C}$, so $G\cdot {\mathcal G}(\tuple{g})\preceq G\cdot H$ (Theorem~\ref{thm:genericstab_conn}(b)).  It follows from Lemma~\ref{lem:transitive} that $G\cdot K_1\preceq G\cdot H$.
 
 We have $G\cdot H= {\mathcal D}(K^0)$ by choice of $v$ and Theorem~\ref{thm:genericstab_conn}, so $G\cdot H\preceq G\cdot K^0$ by Lemma~\ref{lem:preccrit}.  Now $G\cdot K^0\preceq G\cdot K_1$ as $K^0\leq K_1$, so $G\cdot H\preceq G\cdot K_1$ by Lemma~\ref{lem:transitive}.  It follows from Lemmas~\ref{lem:preccrit} and \ref{lem:antisymm} that $G\cdot H= {\mathcal D}(K_1)$.  Now ${\mathcal D}(K_1)= G\cdot M_1$ as $M_1$ is reductive and ${\rm char}(k)= 0$, so $G\cdot H= G\cdot M_1$.  But this is impossible as $H$ is connected and $M_1$ is not.  We conclude that $F_v= (G_v^0)^N$ after all.  The result now follows.
\end{proof}

We have seen that we obtain stronger results if we know that generic stabilisers (or their identity components) are $G$-cr.  Reductive subgroups are always $G$-cr in characteristic 0, but things are more complicated in positive characteristic.  Our next result shows that if this $G$-complete reducibility condition fails for connected stabilisers then it fails badly: we prove that if there exists $v\in V_0$ such that $G_v^0$ is reductive but not $G$-cr then generic elements of $V$ have the same property.

\begin{prop}
\label{prop:genericnoncr_conn}
 Let $H$ be as in Theorem~\ref{thm:genericstab_conn}.  Let
 $$ \widetilde{B}_H'= \{v\in V_{\rm min}\mid \mbox{$G_v^0$ is not $G$-cr}\}. $$
 If $\widetilde{B}_H'$ is nonempty then $\widetilde{B}_H'$ has nonempty interior.
\end{prop}

\begin{proof}
 Note that ${\mathcal D}(G_v^0)= G\cdot H$ for all $v\in \widetilde{B}_H'$, by Theorem~\ref{thm:genericstab_conn}.  The argument below shows that $\widetilde{B}_H'= V_{\rm min}\cap \widetilde{\eta}(\widetilde{\phi}^{-1}(U))$, where $U$ is the open set defined below, so $\widetilde{B}_H'$ is constructible.  It follows as in the proof of Theorem~\ref{thm:genericstab_conn} that we can extend the ground field $k$; hence we can assume $k$ is solid.
  
 Suppose $\widetilde{B}_H'$ is nonempty.  Let $v\in \widetilde{B}_H'$.  We can choose $\tuple{g}\in (G_v^0)^N$ such that $\alpha_{G_v^0}({\mathcal G}(\tuple{g}))= G_v^0/R_u(G_v^0)$ (Proposition~\ref{prop:topfg}) and such that $1\neq g_N\in R_u(G_v^0)$ if $G_v^0$ is non-reductive.  This ensures that ${\mathcal G}(\tuple{g})$ is not $G$-cr.  As $H$ is $G$-cr and ${\mathcal D}(G_v^0)= G\cdot H$, there exists $\lambda\in Y(G)$ such that $G_v^0\leq P_\lambda$ and $c_\lambda(G_v^0)$ is conjugate to $H$.  Then $c_\lambda({\mathcal G}(\tuple{g}))= c_\lambda(G_v^0)$ is conjugate to $H$, so ${\rm dim}(G_{\tuple{g}})= {\rm dim}(C_G({\mathcal G}(\tuple{g})))< {\rm dim}(C_G(H))$, since ${\mathcal G}(\tuple{g})$ is not conjugate to $H$ (as ${\mathcal G}(\tuple{g})$ is not $G$-cr).  Consider $G^N$ regarded as a $G$-variety.  Let $U$ be the set of all $\tuple{m}\in G^N$ such that ${\rm dim}(G_\tuple{m})< {\rm dim}(C_G(H))$; then $U$ is an open neighborhood $U$ of $\tuple{g}$, by Lemma~\ref{lem:stabdimcty}.
 
 Let $E= \{(w,\tuple{g})\in \widetilde{C}\cap \widetilde{\phi}^{-1}(U)\cap \widetilde{\eta}^{-1}(V_{\rm min})\mid \tuple{g}\in (G_w^0)^N\}$.  By Lemma~\ref{lem:cpts}(b), $E$ is dense in $\widetilde{C}$, so $\widetilde{\eta}(E)$ is dense in $V$.  To complete the proof, it is enough to show that $\widetilde{\eta}(E)\subseteq \widetilde{B}_H'$: for then $\widetilde{B}_H'$, being constructible and dense, has nonempty interior.  So let $w\in \widetilde{\eta}(E)$.  Pick $\tuple{m}$ such that $(w,\tuple{m})\in E$.  Then $\tuple{m}\in (G_w^0)^N\cap U$, so ${\rm dim}(C_G(\tuple{m}))< {\rm dim}(C_G(H))$, so ${\rm dim}(C_G(G_w^0))< {\rm dim}(C_G(H))$ also.  It follows by running the argument above for $G_v^0$ in reverse that $G_w^0$ is not $G$-cr.  Hence $w\in \widetilde{B}_H'$, as required.
\end{proof}

\begin{rem}
\label{rem:notgeneric}
 A similar argument establishes the following.  Let $H$ be as in Theorem~\ref{thm:genericstab}.   If there exists $(v,\tuple{g})\in C$ such that $v\in V_0$, ${\mathcal D}(G_v)= G\cdot H$ and $G_v$ is not $G$-cr then there is an open neighbourhood $U$ of $(v,\tuple{g})\in C$ such that for all $(w,\tuple{g'})\in U$, $G_w$ is not $G$-cr.  But this does not yield an analogue of Proposition~\ref{prop:genericnoncr_conn} for $G_v$ (see Example~\ref{exmp:cosets}(b))---the problem is that $\eta(U)$ need not be dense in $V$.
\end{rem}

\section{Examples}
\label{sec:ex}

In this section we present some examples that show the limits of our results and illustrate some of the phenomena that can occur.  We assume $N\geq \Theta+ 1$.

\begin{exmp}
\label{exmp:cosets}
 We consider a special case of the set-up from the proof of Theorem~\ref{thm:subgpint}.  Let $G= {\rm PGL}_2(k)$, let $M\leq G$ and let $V$ be the quasi-projective variety $G/M$ with $G$ acting by left multiplication.  We assume that $M\cap gMg^{-1}= 1$ for generic $g\in G$ (this will hold in all the cases we consider).  Then $G_w= 1$ for generic $w\in V$, so the subset $H$ from Theorem~\ref{thm:genericstab} is $1$, and $\widetilde{C}= V\times \{\tuple{1}\}$ by Corollary~\ref{cor:finstab}.  In particular, $F_w= \{\tuple{1}\}$ for all $w\in V$.  Let $v= M\in G/M$; then $G_v= M$.\smallskip\\
 \noindent (a) Let $M$ be a maximal torus of $G$.  Then $F_v$ is properly contained in $M^N$, so we see that $F_v$ need not contain all of $(G_v^0)^N$ when $v\not\in V_0$  (cf.\ Remark~\ref{rem:pathologies}).   The subset $B_H$ is dense but not closed in $V$, as ${\mathcal D}(G_v)= G\cdot M$.\smallskip\\
 \noindent (b) Let $M= \langle x\rangle$, where $x\in G$ is a nontrivial unipotent element.  Then $V= V_0= V_{\rm red}$ and $G_w$ is unipotent for all $w\in V$, so ${\mathcal D}(G_v)= \{1\}$ for all $w\in V$ (where $1$ denotes the trivial subgroup).  Now $G_w= 1$ is $G$-cr for generic $w\in V$ but $G_v$ is not $G$-cr.  Hence the set $\{w\in V_{\rm red}\mid {\mathcal D}(G_w)= G\cdot H\ \mbox{and $G_v$ is not $G$-cr}\}$ is nonempty but not dense in $V$ (cf.\ Remark~\ref{rem:notgeneric}).  The irreducible components of $C$ apart from $\widetilde{C}$ do not dominate $V$.
 \smallskip\\
 (c) Let $M= {\rm PGL}_2(q)$, where $q$ is a power of the characteristic $p$.  We have $V= V_0= V_{\rm red}$.  Now $M$ is $G$-ir, so the set $\{w\in V_{\rm red}\mid G_w\ \mbox{is $G$-ir}\}$ is nonempty but not dense in $V$.  Moreover, the set $O$ from Theorem~\ref{thm:genericstab} does not contain the whole of $V_{\rm red}$.
\end{exmp} 

\begin{exmp}
\label{exmp:unipotent}
 Suppose $G$ is connected and not a torus.  Let $m\in \NN$ and let $V$ be the variety of $m$-tuples of unipotent elements of $G$, with $G$ acting on $V$ by simultaneous conjugation.  We claim that $\{(1,\ldots, 1)\}\times G^N$ is an irreducible component of $C$.  For let $D$ be an irreducible component of $C$ such that $\{(1,\ldots, 1)\}\times G^N\subseteq D$.  Consider the element $(1,\ldots, 1,\tuple{g})\in D$, where the components of $\tuple{g}\in G^N$ are all regular semisimple elements of $G$.  There is an open neighborhood $O$ of $\tuple{g}$ in $G^N$ consisting of tuples of regular semisimple elements.  If $(v_1,\ldots, v_m,\tuple{g'})\in \phi^{-1}(O)$ then each component of $\tuple{g'}$ is a regular semisimple element of $g$ centralising the unipotent elements $v_1,\ldots, v_m$ of $G$.  But this forces $v_1,\ldots, v_m$ to be 1.  It follows that $D= \{(1,\ldots, 1)\}\times G^N$, as claimed.  Hence $\eta(D)= \{(1,\ldots, 1)\}$ and $\eta(D)\cap V_0$ is empty (note also that if $m$ is large enough then the dimension inequality from Lemma~\ref{lem:cpts}(a) is violated).  We see that the set $\{w\in V\mid G_w^0\ \mbox{is $G$-ir}\}$ is nonempty but not dense in $V$.
 
 It is not hard to show that $F_{(1,\ldots, 1)}\subseteq \{g\cdot \tuple{g}\mid \tuple{g}\in U^N\}$, where $U$ is a maximal unipotent subgroup of $G$; in particular, we see as in Example~\ref{exmp:cosets}(a) that $F_v$ need not contain all of $(G_v^0)^N$ when $v\not\in V_0$.  Moreover, since the centraliser of a nontrivial unipotent subgroup of a connected group can never be reductive, the only reductive stabiliser is $G_{(1,\ldots, 1)}$, so $V_{\rm red}$ is empty.
\end{exmp}

\begin{exmp}
\label{exmp:notprinc}
 Let $X$ be an affine variety and let $M$ be a reductive linear algebraic group.  Suppose we are given a morphism $f\colon X\times M\ra X\times G$ of the form $f(x,m)= (x,f_x(m))$, and suppose further that each $f_x\colon M\ra G$ is a homomorphism of algebraic groups.  Set $K_x= {\rm im}(f_x)$.  Define actions of $G$ and $M$ on $X\times G$ by $g\cdot (x,g')= (x,gg')$ and $m\cdot (x,g')= (x,g'f_x(m)^{-1})$.  These actions commute with each other, so we get an action of $G$ on the quotient space $V:= (X\times G)/M$.
 
 Now suppose moreover that ${\rm dim}(K_x)$ is independent of $x$.  Then the $M$-orbits on $X\times G$ all have the same dimension, so they are all closed.  This means the canonical projection $\varphi$ from $X\times G$ to $V$ is a geometric quotient, so its fibres are precisely the $M$-orbits \cite[Cor.\ 3.5.3]{New}.  A straightforward calculation shows that for any $(x,g)\in X\times G$, the stabiliser $G_{\varphi(x,g)}$ is precisely $gK_xg^{-1}$.  It follows that if $X$ is infinite and the subgroups $K_x$ are pairwise non-conjugate as $x$ runs over the elements of a dense subset of $X$ then $V$ has no principal stabiliser.
 
 Here is a simple example.  Let $G= {\rm SL}_2(k)$, let $X= k$ and let $M= C_p\times C_p= \langle \gamma_1,\gamma_2\mid \gamma_1^p= \gamma_2^p= [\gamma_1,\gamma_2]= 1\rangle$.  Define $f\colon X\times M\ra X\times G$ by $f(x,m)= (x,f_x(m))$, where $f_x(\gamma_1^{m_1}\gamma_2^{m_2}):=
 \left(
\begin{array}{cc}
 1 & m_1x+ m_2x^2 \\
 0 & 1
\end{array}
\right)
$.
It is easily checked that $f$ has the desired properties, so $V:= (X\times M)/G$ has no principal stabiliser.  Note also that generic stabilisers are nontrivial finite unipotent groups, but the element $v= \varphi(0,1)$ has trivial stabiliser.

Here is an example where the stabilisers are connected.  Daniel Lond \cite[Sec.\ 6.5]{lond} produced a family, parametrised by $X:= k$, of homomorphisms from $M:= {\rm SL}_2(k)$ to $G:= B_4$ in characteristic 2 with pairwise non-conjugate images.  Using this one can construct a morphism $f\colon X\times M\ra X\times G$ with the desired properties, giving rise to a $G$-variety $V:= (X\times M)/G$ having no principal stabiliser and with all stabilisers connected and reductive.  Results of David Stewart give rise to a similar construction for $G= F_4$ in characteristic 2 \cite[Sec.\ 5.4.3]{stewart}.
\end{exmp}

\begin{exmp}
\label{exmp:guralnick}
 We now give an example where there is a point with trivial stabiliser but generic stabilisers are finite and linearly reductive, using another special case of the set-up from the proof of Theorem~\ref{thm:subgpint}.  We describe a recipe for producing such examples, given in \cite[Cor.\ 3.10]{BGS}.  Take a simple algebraic group $G$ of rank $s$ in characteristic not 2 and set $M= C_G(\tau)$, where $\tau$ is an involution that inverts a maximal torus of $G$.  Then the affine variety $G/M$, with $M$ acting by left multiplication, has precisely one orbit that consists of points with trivial stabiliser.  Let $V= G/M\times G/M$ with the product action of $G$.  Then generic stabilisers of points in $V$ are 2-groups of order $2^s$, but $V$ contains points with trivial stabiliser.  Thus 
$V= V_0= V_{\rm red}$ and $\widetilde{C}= V\times \{\tuple{1}\}$.  Since 2-groups are linearly reductive---and hence $G$-cr---in characteristic not 2, the $G$-cr subgroup $H$ from Theorem~\ref{thm:genericstab} must be a 2-group of order $2^s$, and moreover, $H$ is a principal stabiliser for $V$ by Corollary~\ref{cor:Gcrprinc}.  The set $\{v\in V_{\rm red}\mid {\mathcal D}(G_v)= G\cdot H\}$ does not contain the whole of $V_{\rm red}$ (cf.\ Theorem~\ref{thm:genericstab_conn}).

 We claim that there is at least one irreducible component $D$ of $C$ such that $\ovl{\eta(D)}= V$ but $\eta(D)\neq V$.  Let $D_1,\ldots, D_t$ be the irreducible components of $V$ apart from $C$.  Then $\bigcup_{i= 1}^t \ovl{\eta(D_i)}= V$, so $\eta(D_j)$ is dense in $V$ for some $1\leq j\leq t$.  There are only finitely many conjugacy classes of nontrivial elements of $G$ of order dividing $2^s$, and each such conjugacy class is closed because in characteristic not 2, elements of order a power of 2 are semisimple.  Hence there are regular functions $f_1, \ldots, f_m\colon G\ra k$ for some $m$ such that for all $g\in G$, $g$ is a nontrivial element of order dividing $2^s$ if and only if $f_1(g)= \cdots = f_m(g)= 0$.  For $1\leq l\leq N$, let $Z_l$ be the closed subset $\{(v,g_1,\ldots, g_N)\in C\mid f_1(g_l)=\cdots = f_m(g_l)= 0\}$ of $C$ and let $Z= Z_1\cup\cdots \cup Z_N$.  If $(v,\tuple{g})\in D_j\backslash (D_j\cap \widetilde{C})$ then $\tuple{g}\neq 1$, so some component of $\tuple{g}$ is a nontrivial element of $G$ of order dividing $2^s$, so $(v,\tuple{g})\in Z$.  Hence the open dense subset $D_j\backslash (D_j\cap \widetilde{C})$ of $D_j$ is contained in $Z$, and it follows that $D_j\subseteq Z$.  This implies that if $v\in V$ and $G_v= 1$ then $v\not\in \eta(D_j)$. 
\end{exmp}


\bigskip
{\bf Acknowledgements}:
The author acknowledges the financial support of
 Marsden Grant UOA1021.  He is grateful to Robert Guralnick for introducing him to the reductivity problem for generic stabilisers and for helpful discussions.  He also thanks V. Popov for drawing his attention to references \cite{rich72acta}, \cite{pop} and \cite{LunaVust}, and the referee for helpful comments.



\begin{thebibliography}{00}
\bibitem{BGS}
T.C.~Burness, R.M.~Guralnick, J.~Saxl,
\emph{On base sizes for algebraic groups}, to appear in J. Eur.\ Math.\ Soc.  {\tt arXiv:1310.1569 [math.GR]}.

\bibitem{BMR}
M.~Bate, B.~Martin, G.~R\"ohrle,
\emph{A geometric approach to complete reducibility},
Invent. Math. \textbf{161}, no. 1 (2005), 177--218.

\bibitem{GIT}
M.~Bate, B.~Martin, G.~R\"ohrle, R.~Tange,
\emph{Closed orbits and uniform $S$-instability in geometric invariant theory},
Trans.\ Amer.\ Math.\ Soc., \textbf{365}  (2013),  no. 7, 3643--3673.

\bibitem{Bo}
A.~Borel,
\emph{Linear algebraic groups},
Graduate Texts in Mathematics, \textbf{126}, Springer-Verlag 1991.

\bibitem{bongartz}
K.~Bongartz,
\emph{Some geometric aspects of representation theory}, In \emph{Algebras and modules, I (Trondheim, 1996)}, 1--27, CMS Conf.\ Proc.\ {\bf 23}, Amer.\ Math.\ Soc., Providence, RI, 1998.

\bibitem{BR}
P.~Bardsley, R.W.~Richardson,
\emph{\'Etale slices for algebraic transformation groups in characteristic $p$}, Proc.\ London Math.\ Soc.\ (3) {\bf 51} (1985), no.\ 2, 295--317.

\bibitem{clo}
D.~Cox, J.~Little, D.~O'Shea,
\emph{Ideals, varieties, and algorithms. An introduction to computational algebraic geometry and commutative algebra}, fourth edition, Undergraduate Texts in Mathematics. Springer, New York, 2015. 

\bibitem{hochschild}
G.P.~Hochschild,
\emph{Basic theory of algebraic groups and Lie algebras}, Graduate Texts in Mathematics {\bf 75}. Springer-Verlag, New York-Berlin, 1981.

\bibitem{Hum}
J.E. Humphreys,
\emph{Linear Algebraic Groups},
Springer-Verlag, New York, 1975.

\bibitem{lond}
D.~Lond,
\emph{On reductive subgroups of algebraic groups and a question of {K}\"ulshammer}, PhD thesis, University of Canterbury, 2013.

\bibitem{LunaVust}
D.~Luna and T. Vust,
\emph{Un th\'eor\`eme sur les orbites affines des groupes alg\'ebriques semi-simples}, Ann.\ Scuola Norm.\ Sup.\ Pisa (3) \textbf{27} (1973), 527--535.

\bibitem{Mar}
B.M.S.~Martin,
\emph{Reductive subgroups of reductive groups in nonzero characteristic},
J. Algebra, \textbf{262} (2003), no.\ 2, 265--286.

\bibitem{New}
P.E. Newstead,
\emph{Introduction to moduli problems and orbit spaces},
Tata Institute of Fundamental Research Lectures on Mathematics and Physics, \textbf{51}, Tata Institute of Fundamental Research, Bombay 1978.

\bibitem{nisnevic}
E.A.~Nisnevi{\v{c}},
\emph{Intersection of the subgroups of a reductive group, and stability of the action}, Dokl.\ Akad.\ Nauk BSSR {\bf 17} (1973), 785--787, 871.

\bibitem{pop}
V.L.~Popov,
\emph{On the stability of the action of an algebraic group on an algebraic variety}, 1972 Math.\ USSR Izv.\ \textbf{6}, 367--379.

\bibitem{rich2}
R.W.~Richardson,
\emph{Conjugacy classes in Lie algebras and algebraic groups},
Ann.\  Math.\  \textbf{86}, (1967), 1--15.

\bibitem{rich72}
R.W.~Richardson,
\emph{Principal orbit types for algebraic transformation spaces in characteristic zero},
Invent.\ Math.\ \textbf{16} (1972), 6--14.

\bibitem{rich72acta}
R.W.~Richardson,
\emph{Deformations of Lie subgroups and the variation of isotropy subgroups}, Acta Math.\ \textbf{129} (1972), 35--73.

\bibitem{Rich77}
R.W.~Richardson,
\emph{Affine coset spaces of reductive algebraic groups},
Bull.\ London Math.\ Soc.\ \textbf{9} (1977), no.\ 1, 38--41.

\bibitem{rich}
R.W.~Richardson,
\emph{Conjugacy classes of $n$-tuples in Lie algebras and algebraic groups},
Duke Math. J. \textbf{57} (1988), no.\ 1, 1--35.

\bibitem{serre1.5}
J-P. Serre,
\emph{La notion de compl\`ete r\'eductibilit\'e
dans les immeubles sph\'eriques et les groupes r\'eductifs},
S\'eminaire au Coll\`ege de France, r\'esum\'e dans
\cite[pp.\ 93--98]{tits2} (1997).

\bibitem{serre1}
\bysame,  
\emph{The notion of complete reducibility in group theory},
Moursund Lectures, Part II, University of Oregon, 1998,\
{\tt arXiv:math/0305257v1 [math.GR]}.

\bibitem{serre2}
\bysame,
\emph{Compl\`ete r\'eductibilit\'e},
S\'eminaire Bourbaki, 56\`eme ann\'ee, 2003-2004, n$^{\rm o}$ 932.

\bibitem{sikora}
A.S.~Sikora, \emph{Character varieties},
Trans.\ Amer.\ Math.\ Soc. {\bf 364} (2012), no.\ 10, 5173--5208.

\bibitem{spr2}
T.A. ~Springer, \emph{Linear algebraic groups},
Second edition. Progress in Mathematics, 9. Birkh\"auser Boston, Inc.,
Boston, MA (1998).

\bibitem{stewart}
D.I.~Stewart,
\emph{$G$-complete reducibility and the exceptional algebraic groups}, PhD thesis, Imperial College London, 2010.

\bibitem{tits2}
J. Tits,
\emph{Th\'eorie des groupes},
R\'esum\'e des Cours et Travaux, 
Annuaire du Coll\`ege de France, $97^{\rm e}$ ann\'ee,
(1996--1997), 89--102.

\bibitem{wallach}
N. Wallach, private communication.


\end{thebibliography}
\end{document}